\documentclass[12pt]{amsart}
\usepackage{epsfig}
\usepackage{graphics}
\usepackage{amscd}
\numberwithin{equation}{section}
\theoremstyle{plain}
\newtheorem{thm}{Theorem}[section]

\newtheorem{prop}[thm]{Proposition}
\newtheorem{lemma}[thm]{Lemma}

\theoremstyle{definition}
\newtheorem{dfn}[thm]{Definition}
\newtheorem{ex}[thm]{Example}
\theoremstyle{remark}
\newtheorem{rem}[thm]{Remark}

\newcommand{\C}{\mathbb C}
\newcommand{\N}{\mathbb N}

\newcommand{\R}{\mathbb R}

\newcommand{\Z}{\mathbb Z}

\def\rk{\operatorname{rank}}
\def\dim{\operatorname{dim}}
\def\grad{\operatorname{grad}}

\begin{document}
\title[On linear deformations of Brieskorn singularities]{On linear deformations of Brieskorn singularities of two variables into generic maps}
\author[K.Inaba]{Kazumasa Inaba}
\address[K.Inaba]{Mathematical Institute, Tohoku University, Sendai, 980-8578, Japan}
\email{sb0d02@math.tohoku.ac.jp }
\author[M.Ishikawa]{Masaharu Ishikawa}
\address[M.Ishikawa]{Mathematical Institute, Tohoku University, Sendai, 980-8578, Japan}
\email{ishikawa@m.tohoku.ac.jp}
\author[M.Kawashima]{Masayuki Kawashima}
\address[M.Kawashima]{Department of Mathematics, Tokyo University of Science, 1-3 Kagurazaka, Shinjuku-ku, Tokyo 
162-8601, Japan}
\email{kawashima@ma.kagu.tus.ac.jp}
\author[T.T.Nguyen]{Tat Thang Nguyen}
\address[T.T.Nguyen]{Institute of Mathematics, Vietnam Academy of Science and Technology, 10307 Hanoi, Vietnam}
\curraddr{Mathematical Institute, Tohoku University, Sendai, 980-8578, Japan}
\email{ntthang@math.ac.vn}

\begin{abstract}
In this paper, we study deformations of Brieskorn polynomials of two variables 
obtained by adding linear terms consisting of the conjugates of complex variables
and prove that the deformed polynomial maps have only indefinite fold and cusp singularities in general.
We then estimate the number of cusps appearing in such a deformation.
As a corollary, we show that a deformation of a complex Morse singularity with real linear terms
has only indefinite folds and cusps in general and the number of cusps is $3$.
\end{abstract}

\maketitle

\section{Introduction}

Let $f:\C^2\to\C$ be a complex polynomial map of two variables $(z,w)$ with isolated singularity at the origin
and $f_t:\C^2\to\C$, $t\in [0,1]$, be a family of polynomial maps such that $f_0=f$, i.e., a deformation of $f$.
A deformation of $f$ decomposes the singularity of $f$ at the origin into several smaller singularities,
especially into complex Morse singularities for a generic choice of the deformation.
Such a deformation is useful to understand the topology of the singularity of $f$ at the origin.
For example, from the deformations in~\cite{acampo, acampo2, GZ, GZ2, GZ3} 
we can describe the Dynkin diagram of the singularity and obtain the monodromy matrix of its Milnor fibration explicitly.

We now regard $f$ as a real polynomial map. 
A smooth map is called a {\it generic map} if it is transversal, in the jet-space, to certain submanifolds 
called Thom-Boardman singularities~\cite{boardman}. 
Usually, a normal crossing condition (Condition NC in~\cite[p. 157]{golubitsky}) is required in addition.
It is known that the set of generic maps is a dense subset in the space of smooth maps if the source manifold is compact,
see~\cite{mather2}.
In the case of smooth maps from $\R^4$ to $\R^2$, a smooth map is generic
with ignoring the normal crossing condition
 if and only if it has only fold and cusp singularities.
In this paper, we are interested in deforming a smooth map into a generic one in this sense.
To avoid confusion, we call such a map an {\it excellent map}, which is a terminology used in~\cite{whitney}
for maps between 2-planes, and also used for maps from higher dimensional manifolds later, e.g. in~\cite{saeki}.

We say that a deformation of $f$ is {\it linear} if it is given in the form
\begin{equation}\label{eq_intro}
   f_t(z,w)=f(z,w)+a_1z+b_1w+a_2\bar z+b_2\bar w,
\end{equation}
where $a_1,b_1,a_2,b_2$ are analytic functions with variable $t\in\R$ which vanish at $t=0$.
The variables $\bar z$ and $\bar w$ are the complex conjugates of $z$ and $w$ respectively, which are needed since we
are studying real deformations.
In this paper, we study linear deformations of Brieskorn polynomials of form~\eqref{eq_intro} with $a_1\equiv b_1\equiv 0$.

%

\begin{thm}\label{thm1}
Let $f(z,w)=z^p+w^q$ be a Brieskorn polynomial with $p,q\geq2$.
If $a,b\in\C$ are generic then there exists a linear deformation $f_t(z,w)$ of $f$
which consists of excellent maps for $t\in (0,1]$
and satisfies $f_1(z,w)=f(z,w)+a\bar z+b\bar w$.
Moreover, such an excellent map has only indefinite fold and cusp singularities.
\end{thm}

In the assertion, ``$a,b\in\C$ are generic'' means that $a,b\in\C$ lie outside the union of 
zero-sets of a finite number of equations of the polar coordinates of $a$ and $b$.
See Remark~\ref{generic_ab} for more precise explanation.

Next, we observe the number of cusps of an excellent map obtained as a linear deformation of $f(z,w)=z^p+w^q$
in Theorem~\ref{thm1}, whose existence is confirmed in that theorem.

\begin{thm}\label{thm2}
Let $f(z,w)=z^p+w^q$ be a Brieskorn polynomial with $p\geq q \geq 2$.
Suppose that $f(z,w)+a\bar z+b\bar w$ is an excellent map. Then the number of cusps of this map belongs to
$[(p+1)(q-1),\;(p-1)(q+1)]$.
In particular, if $p=q$ then the number of cusps is $p^2-1$.
\end{thm}

There are related results about the number of cusps appearing in the versal deformations of
complex isolated singularities, see~\cite{greuel1, greuel2}.

If $p=q=2$ then equation~\eqref{eq_intro} can be modified as
\[
\left(z+\frac{a_1}{2}\right)^2+\left(w+\frac{b_1}{2}\right)^2+a_2\bar z+b_2\bar w-\frac{(a_1^2+b_1^2)}{4}.
\]
This is in the form $z^2+w^2+a\bar z+b\bar w$ up to translation of the coordinates and
addition of a constant term. 
Therefore, the results in Theorem~\ref{thm1} and~\ref{thm2} hold for ``any'' linear deformations. 

%
From Theorem~\ref{thm2}, we can conclude that
the number of cusps in a  linear deformation of $f(z,w)=z^2+w^2$ into an excellent map is $3$.
Furthermore, in this case, we can explicitly describe the image of the set of singularities of the deformed map as follows.

\begin{thm}\label{cor4}
Let $f_t$ be a linear deformation of $f(z,w)=z^2+w^2$ into excellent maps. 
Then 
the set of singular values of $f_t$, $t\in (0,1]$, in $\R^2$ is a scaling and rotation of the curve 
given by the map $h:S^1\to\R^2$ defined by
\[
   h(\theta)=e^{2i\theta}+2e^{-i\theta}.
\]
This curve is a simple closed $3$-cuspidal curve as shown in Figure~\ref{fig0}.
The map $f_t$, $t\in (0,1]$, restricted to the set of singularities is injective and hence the number of cusps is $3$.
\end{thm}

\begin{figure}[htbp]
\begin{center}
  \includegraphics[scale=1.0]{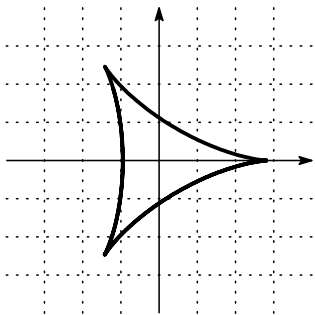}
  \caption{The curve given by $h(\theta)=e^{2i\theta}+2e^{-i\theta}$.\label{fig0}}
\end{center}
\end{figure}



We will use the higher differentials of H.~Levine in~\cite{levine} to prove Theorem~\ref{thm1},
which give a criterion to determine if a smooth map is an excellent map or not.
%
It is important to remark that the mappings studied in Theorem~\ref{thm1} are improper and
we do not know what happens near the infinity. Therefore the ordinary stability analysis of differentiable mappings 
cannot be applied directly. 
It is very hard to apply Levine's criterion to polynomial maps which are explicitly given.
We could do this because we know, in advance, that their cusps are determined by the equation $T=0$ which will be explained 
in Section~4. (Compare this equation with $\phi(z_0)$ in Proposition~\ref{prop38}.)
It may be possible to show the same claim using other techniques, but, since the equation $T=0$ 
contains trigonometric functions, we guess that the difficulty of the calculation 
will be more or less the same.
In this sense, our direct calculation performed to prove the theorem is indispensable.
The assumption that the polynomial is of Brieskorn type is necessary to perform our calculation,
otherwise the defining equation of cusps will become more complicated so that we cannot apply Levine's criterion practically.

Remark also that the topological studies of broken Lefschetz fibrations, initiated by 
the works of Auroux-Donaldson-Katzarkov~\cite{adk} and O.Saeki~\cite{saeki},
is related to our interest.
For example, the linear deformation in Theorem~\ref{cor4} appears in~\cite[Move 4 in p.292]{lekili}
as a typical move which produces a broken Lefschetz fibration.
The observation in this paper shows concretely that such a linear deformation is really an excellent map 
for a generic choice of the parameters.

In the recent work~\cite{DF} of N.~Dutertre and T.~Fukui,
they gave several formulas of Euler characteristics for singularities of stable maps 
under the assumption that the maps are locally trivial at infinity.
Especially, it follows from their result that 
the number of cusps of stable and locally trivial at infinity maps appearing in a small deformation in our setting
is constant modulo $2$.
This phenomenon can be seen in examples given in the last section in this paper.

The paper is organized as follows. Section~2 is for preliminaries.
We briefly introduce the higher differentials of H.~Levine in~\cite{levine} in Section~2.1, 
and then explain how to use it in Section~2.2. This is the main technique in this paper,
and the rest is full of hard calculation.
We then prove the first assertion in Theorem~\ref{thm1} in Section 3,
Theorem~\ref{thm2} and the second assertion in Theorem~\ref{thm1} in Section 4, 
and finally prove Theorem~\ref{cor4} in Section~5.
In Section 6, we close the paper with presenting two examples of linear deformations.

The second and last authors are supported by the Japan Society for the Promotion of Science (JSPS) 
Postdoctoral Fellowship for Foreign Researchers' Grant-in-Aid 25/03014.
The last author is also partially supported by Vietnam National Foundation for Science and Technology Development (NAFOSTED) 
grant 101.04-2014.23 and the Vietnam Academy of Science and Technology (VAST).


\section{Preliminaries}

\subsection{Higher differentials}

Let $f: X\to Y$ be a smooth map between an $n_X$-dimensional manifold $X$ and 
an $n_Y$-dimensional manifold $Y$.
Let $j^rf(p)$ denote the $r$-jet of $f$ at $p\in X$ and let $J^r(f):X\to J^r(X,Y)$ be the map
which maps $p\in X$ to $j^rf(p)$, where $J^r(X,Y)$ is the set of $r$-jets.
Denote by $df$ the induced map from $TX$ to $f^{-1}(TY)$ and $df_p=df|{T_pX}$ for $p\in X$,
where $TX$ is the tangent bundle of $X$,
$T_pX$ is the tangent space of $X$ at $p$ and
$f^{-1}(TY)$ is the vector bundle over $X$ whose fiber at $p\in X$
is $T_{f(p)}Y$. 

In this section, we will introduce higher differentials of $f$ according to~\cite{levine},
by which we interpret the transversality of $J^1(f)$ with $S_i(X, Y)$ and of $J^2(f)$ with $S_{i,h}(X, Y)$.
Here $S_i(X,Y)$ is defined as
\[
   S_i(X,Y)=\{j^1f(p)\in J^1(X,Y)\mid \rk df_p=\min\{n_X, n_Y\}-i\}.
\]
For the definition of $S_{i,h}(X, Y)$, see \cite{levine2} and also \cite{golubitsky}.
We define
\[
   S_i(f)=\{p\in X \mid \rk df_p= \min(n_X, n_Y)- i\}.
\]
Note that $S_0(f)$ is the set of regular points of $f$ and $S(f)= \cup _{i\geq 1}S_i(f)$ is the set of singular points of $f$,
i.e., the points at which the differential of $f$ is not surjective.
Note also that the sets $S_{i_1,i_2,\cdots,i_k}(f)$ play important role in the study of stable/generic maps.
The sequence $\{i_1,i_2,\cdots,i_k\}$ is called the {\it Boardman symbol of order $k$}.
In this paper, since we are interested in determining if a singularity is a fold or a cusp, 
we only need the Boardman symbols up to order $3$.
The set $S_{i,h}(f)$ will be introduced after Proposition~\ref{prop22}
and the set $S_{i,h,k}(f)$ will be introduced in the end of this subsection.


For $p\in X$ and $q\in Y$, let $U$ and $V$ be coordinate neighborhoods of $p$ and $q$, respectively, such that 
$f(U)\subset V$.
Since $TX|U$ and $TY|V$ are trivial we can choose bases $\{u_i\}$ and $\{v_j\}$ of 
the sections of these restricted bundles. Let $\{u_i^{*}\}$ and $\{v_j^{*}\}$ be the associated dual bases
\[
\langle u_i, u_{i'}^*\rangle= \delta_{ii'},\quad \langle v_j, v_{j'}^*\rangle= \delta_{jj'},
\]
where $\langle\;,\;\rangle$ is the pairing of a vector space with its dual.

Let $w_j= v_j\circ f$ and $w^*_j= v^*_j\circ f$. 
Since $df$ is linear on each fiber, there are smooth functions $a_{ij}, i=1,\ldots, n_X$, $j=1,\ldots,n_Y$, 
on $U$ such that
\[
df=\sum_{i,j}a_{ij}u^*_i\otimes w_j,
\]
where
\[
(df(u_i))_p=\sum_{j=1}^{n_Y}a_{ij}(p)w_j(p).
\]

Set $E=TX|S_i(f)$ and $F= f^{-1}(TY)|S_i(f)$. 
Define $L$ and $G$ by the exactness of the sequence
\[
0 \longrightarrow L\longrightarrow E \overset{df}{\longrightarrow} F \overset{\pi_1}{\longrightarrow} G \longrightarrow 0.
\]
Note that $L$ (resp. $G$) is an $(n_X-k)$-dimensional (resp. $(n_Y-k)$-dimensional)
vector bundle over $S_i(f)$, where $k=\min\{n_X, n_Y\}-i$.
We denote the fibers of $L$ and $G$ at $p\in S_i(f)$ by $L_p$ and $G_p$ respectively.
Define the map $\varphi^1: E\to L^* \otimes F$, for each $p\in S_i(f)$, by
\begin{equation}\label{formula20}
\varphi^1_p(v)(l)=\sum_{i,j}\langle v, da_{ij}(p)\rangle\langle l, u^*_i(p)\rangle w_j(p),
\end{equation}
with $v\in T_pX$ and $l\in L_p$.
Then the second differential $d^2f: E\to L^* \otimes G$ of $f$ is defined as 
$d^2f_p(v)(l)= \pi_1(\varphi^1_p(v)(l))$.

\begin{prop}[\cite{levine}]
The map $d^2f$ is a well-defined bundle map and $d^2f|L$ is a section in $(L^* \odot L^*)\otimes G$, 
where $\odot $ represents the symmetric product.
\end{prop}

The transversality of $J^1(f)$ to $S_i(X, Y)$ is determined by $d^2f$ as follows.

\begin{prop}[\cite{levine2, levine}]\label{prop22}
The map $J^1(f)$ is transversal to $S_i(X, Y)$ at $p\in S_i(f)$ if and only if the map
\[
d^2f_p: T_pX\to L_p^*\otimes G_p
\]
has rank $(n_X-k)(n_Y-k)= i(|n_X-n_Y|+i)$.
\end{prop}

Let $\psi$ be the restriction of $d^2f_p$ to $L_p$, 
then $\psi$ is represented by an $(n_X-k)\times (n_X-k)(n_Y-k)$ matrix. 
Assume that $J^1(f)$ is transversal to $S_i(X,Y)$ at $p\in S_i(f)$.
Denote
\[
   S_{i,h}(f)=\{p\in X\mid \rk \psi=n_X-k-h\}.
\]
Note that if $p\in S_{i,0}(f)$ then $p$ is a regular point of $f$ restricted on $S_i(f)$.
Using local coordinates we can check that
in a neighborhood of $p$ the set $S_i(f)$ is a submanifold of codimension $(n_X-k)(n_Y-k)$.
Let $U$ denote the neighborhood of $p$. Define the vector bundle $T$ on $S_i(f)\cap U$
by the exactness of the following sequence of the vector bundles on $S_i(f)\cap U$:
\[
0\longrightarrow T\longrightarrow E \overset{d^2f}{\longrightarrow} L^*\otimes G\longrightarrow 0.
\]
We can prove that $T$ is the tangent bundle to $S_i(f)\cap U$.

The third differential is defined on the subset $S_{i,h}(f)$. 
Assume that $J^1(f)$ is transversal to $S_i(X, Y)$. 
Define the vector bundles $R$ and $K$ on $S_{i, h}(f)$ by the exact sequence
\[
0\longrightarrow R\longrightarrow L 
\overset{d^2f}{\longrightarrow} L^*\otimes G \overset{\pi_2}{\longrightarrow} K\longrightarrow 0,
\]
where $\dim R_p=h$ and $\dim K_p= h(n_Y-k)+ (n_X-k-h)(n_Y-k-1)$.
Let $R^{\bot}$ be the annihilator of $R$ in $L^*$. 
Since $R$ is a subbundle of $L$ we have the exact sequence
\[
0\longrightarrow R^{\bot}\longrightarrow L^*\longrightarrow R^*\longrightarrow 0.
\]
Hence
\[
0\longrightarrow R^{\bot}\otimes G\longrightarrow L^*\otimes G\longrightarrow R^*\otimes G\longrightarrow 0
\]
is also exact. It is easy to check that $d^2f(L)\subset R^{\bot}\otimes G$. Then we have the commutative exact diagram
\[
\begin{CD}
L @>{d^2f}>> L^*\otimes G @>{\pi_2}>> K @>>> 0 \\
@VV{d^2f}V @VV{=}V @VV{\pi}V \\
R^{\bot}\otimes G @>>> L^*\otimes G @>>> R^*\otimes G @>>> 0.
\end{CD}
\]
The maps in the second square are surjective. That implies that
\[
R^*\otimes K\longrightarrow R^*\otimes R^*\otimes G\longrightarrow 0
\]
is exact. Consequently, we can define the bundle $R^*\odot K$ such that the sequence
\[
0\longrightarrow R^*\odot K\longrightarrow R^*\otimes K \overset{\gamma}{\longrightarrow} 
R^*\wedge R^*\otimes G\longrightarrow 0
\]
is exact, where $\gamma$ is defined by the composition
\[
R^*\otimes K\longrightarrow R^*\otimes R^*\otimes G\longrightarrow R^*\wedge R^*\otimes G.
\]

Now, recall that we have the projection maps
\[
   \pi_1: F\to G,\quad \pi_2: L^*\otimes G\to K,
\]
which induce
\[
\begin{split}
\pi_1'&: R^*\otimes L^*\otimes F\to R^*\otimes L^*\otimes G,\\
\pi_2'&: R^*\otimes L^*\otimes G\to R^*\otimes K.
\end{split}
\]
Let $\{z_1, \ldots, z_{n_X-k}\}$ be a part of a basis of sections in $E$ over the neighborhood $U$ of $p$ 
such that $\{z_j\mid S_i(f)\cap U\}$ is a basis for $L|U\cap S_i(f)$. We define $\varphi^2: T\to R^*\otimes L^*\otimes F$ by
\begin{equation}\label{formula21}
\varphi^2_x(t)(r\otimes l)= \sum_{i,j,m}\langle t, d(\langle z_i, da_{mj}\rangle)(x)\rangle
\langle l, z^*_m(x)\rangle\langle r, z_i^*(x)\rangle w_j(x)
\end{equation}
where $x\in U\cap S_i(f)$, $r\in R$ and $l\in L$.
Then the third differential $d^3f: T\to R^*\otimes K$ of $f$ is defined as $d^3f_x= \pi_2'\circ \pi_1'\circ \varphi^2_x$.

\begin{prop}[\cite{levine}]
The map $d^3f$ is well-defined over $S_{i,h}(f)$ and satisfies $d^3f(T)\subset R^*\odot K$.
\end{prop}

The following fact can be proved by calculating the differential of the jet map 
$J^2f: X\to J^2(X, Y)$ and checking the transversal condition of $J^2(f)$ with $S_{i,h}(X, Y)$ 
(see \cite{levine} and \cite{golubitsky}).

\begin{prop}
The map $J^2(f)$ is transversal to $S_{i,h}(X, Y)$ at $p$ if and only if $d^3f_p:T_p\to R^*_p\odot K_p$ is surjective.
\end{prop}

If $J^2(f)$ is transversal to $S_{i,h}(X,Y)$ then $S_{i,h,k}(f)$ is defined as
\[
   S_{i,h,k}(f)=\{p\in S_{i,h}(f)\mid \text{$d(f|_{S_{i,h}(f)})$ drops rank by $k$}\},
\]
see~\cite[p.156]{golubitsky}.

\subsection{Excellent maps}


Suppose that $\dim X=4$ and $\dim Y=2$.
In this setting, an excellent map is defined as follows.

\begin{dfn}
A smooth map $f:X\to Y$ is called an {\it excellent map} if
for each point $p\in X$, there exist a system of local coordinates $x_1,x_2,x_3,x_4$ centered at $p$ 
on $X$ and a system of local coordinates $y_1, y_2$ centered at $f(p)$ on $Y$ such that
$f$ is locally described in one of the following form:
\begin{itemize}
\item[(1)] $(x_1,x_2,x_3,x_4)\mapsto (x_1,x_2)$,
\item[(2)] $(x_1,x_2,x_3,x_4)\mapsto (x_1,x_2^2+x_3^2+x_4^2)$,
\item[(3)] $(x_1,x_2,x_3,x_4)\mapsto (x_1,x_2^2+x_3^2-x_4^2)$,
\item[(4)] $(x_1,x_2,x_3,x_4)\mapsto (x_1,x_2^2\pm x_3^2+x_1x_4+x_4^3)$.
\end{itemize}
The point in case (1) is a regular point. The point in case (2), (3) and (4) is
called a {\it definite fold}, an {\it indefinite fold} and a {\it cusp}, respectively.
\end{dfn}


Excellent maps actually exist ``generically'' in $C^\infty(X,Y)$
because of the following interpretation in terms of transversality.
Let $S_1^2(X,Y)$ be the set of $j^2f(p)\in J^2(X,Y)$ such that $j^1(f)(p)\in S_1(X,Y)$,
$j^1(f)$ is transversal to $S_1(X,Y)$ at $p$, and 
$\rk d(f|_{S_1(f)})=\min\{\dim S_1(f), \dim Y\}-1$, which is the notation in~\cite{levine}.
This set is nothing but $S_{1,1}(X,Y)$ in Section~2.1.
Since we are studying the case where $\dim X=4$ and $\dim Y=2$, we have $\rk d(f|_{S_1(f)})=0$.

\begin{prop}[\cite{levine3, levine}]
A smooth map $f:X\to Y$ is an excellent map if and only if
\begin{itemize}
\item[(a)] $J^1(f)$ is transversal to $S_1(X, Y)$ and $S_2(X, Y)$, and
\item[(b)] $J^2(f)$ is transversal to $S_1^2(X, Y)$.
\end{itemize}
\end{prop}

The condition (a) is equivalent to that 
$d^2f: E\to L^*\otimes G$ is surjective and $S_2(f)= \emptyset$.
The condition (b) is equivalent to that, for any $p\in S(f)$, either (i) $p\in S_{1,0}(f)$ or (ii) $p\in S_{1,1}(f)$ and 
$d^3f_p: T_p\to R^*_p\odot K_p$ is surjective.
Note that the condition (ii) is equivalent to $p\in S_{1,1,0}(f)$.
The following proposition gives a criterion to determine if a singular point is 
a fold, a cusp, or none of them.

\begin{prop}
Let $f:X\to Y$ be a smooth map. Suppose that $p\in S_1(f)$ and $d^2f_p: T_pX\to L_p^*\otimes G_p$ is surjective.
Then $p$ is a fold of $f$ if and only if $p\in S_{1,0}(f)$,
and $p$ is a cusp of $f$ if and only if $p\in S_{1,1}(f)$ and $d^3f_p: T_p\to R^*_p\odot K_p$ is surjective.
\end{prop}

In the proofs of the main theorems, we will explicitly calculate the higher differentials
by choosing suitable basis.
Let $P= (Q,R):\R^4\to \R^2$ be a smooth map and $p\in S_1(P)$. 
We may choose local coordinates $(x_1, x_2, x_3, x_4)$ of $\R^4$ at $p$ 
and those of $\R^2$ at $P(p)$ such that $\grad Q(p)= (1, 0, 0, 0)$ and $\grad R(p)= (0, 0, 0, 0)$. 
We see that $\left\{\frac{\partial }{\partial x_1},\frac{\partial }{\partial x_2},
\frac{\partial }{\partial x_3},\frac{\partial }{\partial x_4}\right\}$
is a basis of sections of $T\R^4$ in 
the neighborhood of $p$.
In this setting, $\dim L_p=3, \dim G_p=1$ and $d^2P_p$ is represented by the matrix
\[
M= \left(\frac{\partial^2 R}{\partial x_i\partial x_j}\right)_{i=1,2,3,4,\;\,j=2,3,4.}
\]
Hence $d^2P_p$ is surjective if and only if $\rk M=3$. We can check that the restriction $d^2P_p|L_p$ is represented by
\[
H= \left(\frac{\partial^2 R}{\partial x_i\partial x_j}\right)_{i=2,3,4, j=2,3,4,}
\]
which is exactly the Hessian of $R$ with variables $(x_2, x_3, x_4)$. 
In particular, $p\in S_{1,1}(P)$ if and only if $\rk H=2.$

If $p\in S_{1,1}(P)$ and $d^2P_p$ is surjective, 
by choosing suitable coordinates $(x_1,x_2,x_3,x_4)$, we may further assume that 
$\frac{\partial^2 R}{\partial x_4\partial x_j}(p)=0$ 
for all $j=2, 3, 4$. Then $\dim L^*_p=3$, $\rk (d^2P_p|L_p)=2$, $\dim R^*_p=1$ and $\dim K_p=1$ imply
$R^*_p\odot K_p= R^*_p\otimes K_p$. Let 
\[
   \xi_j= -\frac{\partial Q}{\partial x_j}\frac{\partial }{\partial x_1}
+\frac{\partial Q}{\partial x_1}\frac{\partial }{\partial x_j}
\]
for $j=2,3,4$.

\begin{lemma}\label{lemma28a}
$\{\xi_2,\xi_3,\xi_4\}$ is a basis of $L|S_1(P)\cap U$ for some neighborhood $U$ of $p$. 
\end{lemma}

\begin{proof}
For $j=2,3,4$, we have $dQ(\xi_j)=0$. Then $dR(\xi_j)=0$ on $S_1(P)$ in this coordinate neighborhood.
Thus $\xi_j$ belongs to $L$.
Since $\frac{\partial Q}{\partial x_1}(p)=1$, they are independent in a small neighborhood of $p$.
\end{proof}

Recall that $\left\{\frac{\partial }{\partial x_2},\frac{\partial }{\partial x_3},\frac{\partial }{\partial x_4}\right\}$ 
is a basis of $L_p$. Since $\rk M=3$ we have $\frac{\partial^2 R}{\partial x_1\partial x_4}(p)\neq 0$. 
Hence $T_p$ is a one-dimensional space generated by $\frac{\partial }{\partial x_4}$. 

In the following, the pairing $\langle\;\;,\;\;\rangle$ is defined with respect to the basis
$\left\{\frac{\partial}{\partial x_1}, \xi_2, \xi_3, \xi_4\right\}$.

\begin{lemma}
The map $d^3P_p: T_p\to R_p^*\otimes K_p$ is given by 
\[
t\frac{\partial}{\partial x_4}
\mapsto t\frac{\partial}{\partial x_4}(\xi_4(\xi_4(R)))(p)\langle \;\;\;, dx_4(p)\rangle w_2(p).
\]
\end{lemma}

\begin{proof}
Since $\frac{\partial Q}{\partial x_1}(p)=1$ and $\frac{\partial Q}{\partial x_j}(p)=0$ for $j=2,3,4$, we have
\[
\langle \;\;\;, \xi_j^*(p)\rangle=
-\frac{\partial Q}{\partial x_j}(p)\langle\;\;\;, dx_1(p)\rangle
+\frac{\partial Q}{\partial x_1}(p)\langle\;\;\;, dx_j(p)\rangle
=\langle\;\;\;, dx_j(p)\rangle
\]
for $j=2,3,4$.
The terms in~\eqref{formula21} with $w_1(p)$ vanish by $\pi_1'$. 
For $j=2,3$, $\langle \;\;\;, \xi_j^*(p)\rangle$ belongs to the image of $d^2P_p$
and therefore their images by the map $\pi_2'$ vanish. Hence we have 
\[
  d^3P_p(t)=\langle t\frac{\partial}{\partial x_4}, d(\langle \xi_4, da_{42}\rangle)(p)\rangle\langle\;\;\;, dx_4(p)\rangle w_2(p).
\]
To prove the assertion, it is enough to show that
$\langle \xi_4, da_{42}\rangle=\xi_4(\xi_4(R))$.
According to the definition of $a_{42}$ we have $a_{42}=\xi_4(R)$.
Thus
$\langle \xi_4, da_{42}\rangle=\langle \xi_4, d(\xi_4(R))\rangle$. 
Since $da_{42}=\frac{\partial a_{42}}{\partial x_1}dx_1+\sum_{j=2,3,4}\xi_j(a_{42})\xi_j^*$, 
we have $\langle \xi_4, da_{42}\rangle=\xi_4(a_{42})=\xi_4(\xi_4(R))$.
This is the equality which we want to have.
\end{proof}

The map $d^3P_p$ is surjective if and only if 
\[
   \frac{\partial}{\partial x_4}(\xi_4(\xi_4(R)))(p)\neq 0.
\]
Thus, in summary, we have the following criterion to check if a 
singularity is a fold, cusp or none of them.

\begin{lemma}\label{lemma28}
In the above setting, $p\in S_1(P)$ is a fold if and only if $\rk H=3$.
It is a cusp if and only if $\rk M=3$, $\rk H=2$ and
$\frac{\partial}{\partial x_4}(\xi_4(\xi_4(R)))(p)\ne 0$.
\end{lemma}

\section{Proof of the first assertion in Theorem~\ref{thm1}}

Throughout this section, for a function $h:\R^n\to\R$ with variables $x_1,\cdots,x_n$,
we denote the partial differentials $\frac{\partial h}{\partial x_i}$,
$\frac{\partial^2 h}{\partial x_i\partial x_j}$ and
$\frac{\partial^3 h}{\partial x_i\partial x_j\partial x_k}$
by $h_{x_i}$, $h_{x_ix_j}$ and $h_{x_ix_jx_k}$ respectively.

We are studying polynomial maps
of the form $f(z,w)=z^p+w^q+a\bar z+b\bar w$.
This polynomial has complex and complex-conjugate variables.
Such a polynomial is called a {\it mixed polynomial} in~\cite{oka}.
Since the assertion in Theorem~\ref{thm1} is for generic $a$ and $b$,
we may assume that $ab\ne 0$.
Let $c_1$ and $c_2$ be non-zero complex numbers satisfying
$c_1^p=a\bar c_1$ and $c_2^q=b\bar c_2$, respectively.
By changing the coordinates as $z=c_1u$ and $w=c_2v$
and setting $\mu=a\bar c_1/(b\bar c_2)$, we have 
\[
\begin{split}
f(z,w)&=(c_1u)^p+a\overline{c_1u}+(c_2v)^q+b\overline{c_2v}\\
&=a\bar c_1(u^p+\bar u)+b\bar c_2(v^q+\bar v) \\
&=b\bar c_2( \mu (u^p+\bar u)+v^q+\bar v).
\end{split}
\]
Now we set
\[
P(u,v;\mu)=\mu(u^p+\bar u)+v^q+\bar v,
\]
with $p,q\geq 2$ and $\mu\in \C\setminus\{0\}$.

\begin{lemma}\label{lemma300}
Suppose that $P$ is an excellent map. Then there exists
a linear deformation $f_t$ of $f$ which consists of excellent maps for $t\in (0,1]$
and satisfies $f_1(z,w)=f(z,w)+a\bar z+b\bar w$.
Moreover, if $P$ has no definite fold then $f_t$ does also for $t\in (0,1]$.
\end{lemma}

\begin{proof}
Set $A(t)=at^{(p-1)q}$ and $B(t)=bt^{p(q-1)}$
and define the linear deformation $f_t(z,w)$ by $f_t(z,w)=f(z,w)+A(t)\bar z+B(t)\bar w$.
Note that this satisfies $f_1(z,w)=f(z,w)+a\bar z+b\bar w$ as required.
Since $c_1(t)^p=A(t)\bar c_1(t)$ and $c_2(t)^q=B(t)\bar c_2(t)$, 
we may choose $c_1(t)=c_{1,0}t^q$ and $c_2(t)=c_{2,0}t^p$,
where $c_{1,0}$ and $c_{2,0}$ are non-zero complex numbers 
satisfying $c_{1,0}^p=a\bar c_{1,0}$ and $c_{2,0}^q=b\bar c_{2,0}$, such that
\[
   \mu=\frac{A(t)\bar c_1(t)}{B(t)\bar c_2(t)}=\frac{at^{(p-1)q}\cdot \bar c_{1,0}t^q}{bt^{p(q-1)}\cdot \bar c_{1,0}t^p}
   =\frac{a\bar c_{1,0}}{b\bar c_{1,0}}.
\]
Therefore, for each $t\in (0,1]$, the smooth map $f_t$
can be replaced by $P$ using this change of coordinates.
\end{proof}


By Lemma~\ref{lemma300}, to prove Theorem~\ref{thm1}, it is enough to show that $P(u,v;\mu)$ is an excellent map 
for a generic choice of $\mu$.
Hereafter we study the map $P$ instead of $f$.

\begin{lemma}\label{lemma31}
A singular point $z_0=(u_0,v_0)$ of $P$ satisfies
\[
\left\{\begin{split}
& p|u_0|^{p-1}=q|v_0|^{q-1}=1,\\
& \frac{p-1}{2}\arg u_0+\arg\mu=\frac{q-1}{2}\arg v_0+\kappa \pi,
\end{split}
\right.
\]
where $\kappa$ is some integer.
\end{lemma}

\begin{proof}
By~\cite[Proposition ~1]{oka},
a singular point $z_0=(u_0,v_0)$ of the mixed polynomial $P$ satisfies
\[
\overline{\frac{\partial P}{\partial u}(z_0)}
=\alpha\frac{\partial P}{\partial \bar u}(z_0) \quad\text{and}\quad
\overline{\frac{\partial P}{\partial v}(z_0)}
=\alpha\frac{\partial P}{\partial \bar v}(z_0)
\]
for a complex number $\alpha$ with $|\alpha|=1$.
Substituting $P(u,v;\mu)=\mu(u^p+\bar u)+v^q+\bar v$
and taking their conjugates, we have 
\[
\mu pu_0^{p-1}=\bar\mu \bar \alpha
\quad\text{and}\quad
qv_0^{q-1}=\bar\alpha.
\]
The first equation in the assertion follows by 
taking the absolute values of these equations,
and the second one follows by 
eliminating $\alpha$ and checking the arguments.
\end{proof}

Set $Q(u,v;\mu)$ and $R(u,v;\mu)$ to be the real and imaginary part of 
$P(u,v;\mu)$ respectively, i.e., $P(u,v;\mu)=Q(u,v;\mu)+iR(u,v;\mu)$.
Set $r_1=|u|$, $\theta_1=\arg u$, $r_2=|v|$ and $\theta_2=\arg v$, so that
$(r_1,\theta_1,r_2,\theta_2)$ are regarded as the polar coordinates of $\C^2$.
Note that since any singular point $z_0=(u_0,v_0)$ satisfies $u_0\ne 0$ and $v_0\ne 0$ by Lemma~\ref{lemma31}
and we only need to observe $P$ in a small neighborhood of the set of singular points,
we may assume that $r_1\ne 0$ and $r_2\ne 0$, that is, $\arg u$ and $\arg v$ are well-defined.
Since
\[
\begin{split}
P&=Q+iR=\mu(u^p+\bar u)+(v^q+\bar v) \\
&=|\mu|r_1^pe^{i(p\theta_1+\theta_\mu)}+|\mu|r_1e^{i(-\theta_1+\theta_\mu)}
+r_2^qe^{iq\theta_2}+r_2e^{-i\theta_2},
\end{split}
\]
we have
\[
\left\{
\begin{split}
Q&=|\mu|r_1^p\cos(p\theta_1+\arg\mu)+|\mu|r_1\cos(-\theta_1+\arg\mu)
+r_2^q\cos(q\theta_2)+r_2\cos(-\theta_2),\\
R&=|\mu|r_1^p\sin(p\theta_1+\arg\mu)+|\mu|r_1\sin(-\theta_1+\arg\mu)
+r_2^q\sin(q\theta_2)+r_2\sin(-\theta_2).
\end{split}
\right.
\]

To simplify the notation, hereafter we may omit $\mu$ in brackets.
For example, we denote $Q(z_0)$ instead of $Q(z_0;\mu)$.

Now we calculate the gradient of the real part $Q$ at 
a singular point $z_0=(u_0,v_0)$ of $P$.
Let $(k_1, k_2, k_3, k_4)$ be the gradient at $z_0$, i.e., 
\[
(k_1, k_2, k_3, k_4)=(Q_{r_1}(z_0), Q_{\theta_1}(z_0),
Q_{r_2}(z_0), Q_{\theta_2}(z_0)).
\]
We set 
\[
\begin{array}{lll}
\displaystyle \Theta_1=\frac{p+1}{2}\arg u_0, & & \displaystyle \Theta_2=\frac{p-1}{2}\arg u_0+\arg\mu, \\
\displaystyle \Theta_3=\frac{q+1}{2}\arg v_0, & & \displaystyle \Theta_4=\frac{q-1}{2}\arg v_0.
\end{array}
\]

\begin{lemma}\label{lemma32}
\[
\begin{array}{lll}
k_1=2|\mu|\cos\Theta_1\cos\Theta_2, & & k_2=-2|\mu||u_0|\sin\Theta_1\cos\Theta_2, \\
k_3=2\cos\Theta_3\cos\Theta_4, & & k_4=-2|v_0|\sin\Theta_3\cos\Theta_4.
\end{array}
\]
\end{lemma}

\begin{proof}
By Lemma~\ref{lemma31},
\[
\begin{split}
k_1&=Q_{r_1}(z_0)=p|\mu||u_0|^{p-1}\cos(p\arg u_0+\arg\mu)+|\mu|\cos(-\arg u_0+\arg\mu) \\
&=|\mu|\{\cos(p\arg u_0+\arg\mu)+\cos(-\arg u_0+\arg\mu)\}\\
&=2|\mu|\cos\left(\frac{p+1}{2}\arg u_0\right)
\cos\left(\frac{p-1}{2}\arg u_0+\arg\mu\right)\\
&=2|\mu|\cos\Theta_1\cos\Theta_2.
\end{split}
\]
The equations for $k_2, k_3$ and $k_4$ are obtained similarly.
\end{proof}

\begin{lemma}\label{lemma33}
The values $\Theta_1, \Theta_2, \Theta_3, \Theta_4$ satisfy 
the following relations:
\[
\left\{\begin{split}
&\Theta_2=\Theta_4+\kappa\pi, \\
&\Theta_1=p\arg u_0-\Theta_2+\arg\mu=\arg u_0+\Theta_2-\arg\mu \\
&\Theta_3=q\arg v_0-\Theta_2+\kappa\pi=\arg v_0+\Theta_2-\kappa\pi, \\
&\frac{p-1}{p+1}\Theta_1+\arg\mu=\frac{q-1}{q+1}\Theta_3+\kappa\pi,
\end{split}
\right.
\]
where $\kappa$ is the integer given in Lemma~\ref{lemma31}.
\end{lemma}

\begin{proof}
The assertion follows from Lemma~\ref{lemma31}
and the defining equations of $\Theta_i$'s.
\end{proof}

\subsection{Fold singularities in case $k_1\ne 0$}

In this subsection, we always assume that $k_1\ne 0$.
The case $k_1=0$ will be studied later.

Let $(r_1',\theta_1', r_2', \theta_2')$ be new coordinates in a small neighborhood of $z_0$ defined by
\[
(r_1',\theta_1', r_2', \theta_2')=
(k_1r_1+k_2\theta_1+k_3r_2+k_4\theta_2, \theta_1, r_2, \theta_2).
\]

\begin{lemma}\label{lemma34}
$(Q_{r_1'}(z_0), Q_{\theta_1'}(z_0), Q_{r_2'}(z_0), Q_{\theta_2'}(z_0))=(1,0,0,0)$.
\end{lemma}

\begin{proof}
Since $r_1=\frac{1}{k_1}r_1'-\frac{k_2}{k_1}\theta_1'-\frac{k_3}{k_1}r_2'-\frac{k_4}{k_1}\theta_2'$
we have $Q_{r_1'}(z_0)=\frac{1}{k_1}Q_{r_1}(z_0)=1$, 
$Q_{\theta_1'}(z_0)=-\frac{k_2}{k_1}Q_{r_1}(z_0)+Q_{\theta_1}(z_0)=0$,
$Q_{r_2'}(z_0)=-\frac{k_3}{k_1}Q_{r_1}(z_0)+Q_{r_2}(z_0)=0$ and
$Q_{\theta_2'}(z_0)=-\frac{k_4}{k_1}Q_{r_1}(z_0)+Q_{\theta_2}(z_0)=0$.
\end{proof}

Set $\hat R=R-s Q$, where $s=R_{r_1'}(z_0)$, so that 
$\hat R_{r_1'}=R_{r_1'}-s Q_{r_1'}$ vanishes at $z_0$.

\begin{lemma}\label{lemma35}
The Hessian $H$ of $\hat R$ with variables
$(\theta_1', r_2', \theta_2')$ is
\[
H=
\begin{pmatrix}
k_2^2A-2k_2B+C & k_3(k_2A-B) & k_4(k_2A-B) \\
k_3(k_2A-B) & k_3^2A+D & k_3k_4A+E \\
k_4(k_2A-B) & k_3k_4A+E & k_4^2A+F
\end{pmatrix},
\]
where
\[
\begin{split}
A&=\frac{1}{k_1^2}\hat R_{r_1r_1},\quad 
B=\frac{1}{k_1}\hat R_{r_1\theta_1},\quad 
C=\hat R_{\theta_1\theta_1} \\
D&=\hat R_{r_2r_2},\quad 
E=\hat R_{r_2\theta_2}, \quad 
F=\hat R_{\theta_2\theta_2}.
\end{split}
\]
Its determinant is 
\[
\det H=(k_4^2D-2k_3k_4E+k_3^2F)(AC-B^2)+(k_2^2A-2k_2B+C)(DF-E^2).
\]
\end{lemma}

\begin{proof}
Since $r_1=\frac{1}{k_1}r_1'-\frac{k_2}{k_1}\theta_1'-\frac{k_3}{k_1}r_2'-\frac{k_4}{k_1}\theta_2'$, we have
\[
\hat R_{\theta_1'}=-\frac{k_2}{k_1}\hat R_{r_1}+\hat R_{\theta_1}, \quad
\hat R_{r_2'}=-\frac{k_3}{k_1}\hat R_{r_1}+\hat R_{r_2}, \quad
\hat R_{\theta_2'}=-\frac{k_4}{k_1}\hat R_{r_1}+\hat R_{\theta_2},
\]
and taking partial differentiation again, we have
\[
\begin{array}{ll}
\displaystyle \hat R_{\theta_1'\theta_1'}
=\frac{k_2^2}{k_1^2}\hat R_{r_1r_1}-\frac{2k_2}{k_1}\hat R_{r_1\theta_1}+\hat R_{\theta_1\theta_1},
& \displaystyle \hat R_{\theta_1'r_2'}=\frac{k_2k_3}{k_1^2}\hat R_{r_1r_1}-\frac{k_3}{k_1}\hat R_{r_1\theta_1},\\
\displaystyle \hat R_{\theta_1'\theta_2'}=\frac{k_2k_4}{k_1^2}\hat R_{r_1r_1}-\frac{k_4}{k_1}\hat R_{r_1\theta_1}, 
& \displaystyle \hat R_{r_2'r_2'}=\frac{k_3^2}{k_1^2}\hat R_{r_1r_1}+\hat R_{r_2r_2}, \\
\displaystyle \hat R_{r_2'\theta_2'}=\frac{k_3k_4}{k_1^2}\hat R_{r_1r_1}+\hat R_{r_2\theta_2},
& \displaystyle \hat R_{\theta_2'\theta_2'}=\frac{k_4^2}{k_1^2}\hat R_{r_1r_1}+\hat R_{\theta_2\theta_2}.
\end{array}
\]
Thus we obtain the Hessian.
The calculation of the determinant is straightforward.
\end{proof}

Now we calculate the Hessian of $\hat R$ at a singular point $z_0=(u_0,v_0)$ of $P$.
We first determine the value $s$ which we used in the above change of coordinates.

\begin{lemma}\label{lemma36}
$s=\tan\Theta_2$.
\end{lemma}

\begin{proof}
By Lemma~\ref{lemma31},
\[
R_{r_1}(z_0)=p|\mu||u_0|^{p-1}\sin(p\arg u_0+\arg\mu)
+|\mu|\sin(\arg\mu-\arg u_0)=2|\mu|\sin\Theta_2\cos\Theta_1.
\]
Hence $s=R_{r_1'}(z_0)=\frac{1}{k_1}R_{r_1}(z_0)=\tan\Theta_2$.
\end{proof}

\begin{lemma}\label{lemma37}
A singular point $z_0=(u_0,v_0)$ of $P$ satisfies
\[
\begin{array}{ll}
\displaystyle A(z_0)=\frac{1}{k_1^2\cos\Theta_2}p(p-1)|\mu||u_0|^{p-2}\sin\Theta_1, 
& \displaystyle B(z_0)=\frac{1}{k_1\cos\Theta_2}(p-1)|\mu|\cos\Theta_1, \\
\displaystyle C(z_0)=-\frac{1}{\cos\Theta_2}(p-1)|\mu||u_0|\sin\Theta_1,
& \displaystyle D(z_0)=\frac{(-1)^\kappa}{\cos\Theta_2}q(q-1)|v_0|^{q-2}\sin\Theta_3, \\
\displaystyle E(z_0)=\frac{(-1)^\kappa}{\cos\Theta_2}(q-1)\cos\Theta_3,
& \displaystyle F(z_0)=\frac{(-1)^{\kappa+1}}{\cos\Theta_2}(q-1)|v_0|\sin\Theta_3.
\end{array}
\]
\end{lemma}

\begin{proof}
By Lemma~\ref{lemma36} and Lemma~\ref{lemma33},
\[
\begin{split}
&A(z_0)=\frac{1}{k_1^2}\hat R_{r_1r_1}(z_0)
=\frac{1}{k_1^2}\left(R_{r_1r_1}(z_0)-sQ_{r_1r_1}(z_0)\right) \\
&=\frac{1}{k_1^2}\left\{
p(p-1)|\mu||u_0|^{p-2}\sin(p\arg u_0+\arg\mu)-sp(p-1)|\mu||u_0|^{p-2}\cos(p\arg u_0+\arg\mu)\right\} \\
&=\frac{1}{k_1^2\cos\Theta_2}
p(p-1)|\mu||u_0|^{p-2}\left\{\sin(p\arg u_0+\arg\mu)\cos\Theta_2
-\cos(p\arg u_0+\arg\mu)\sin\Theta_2\right\} \\
&=\frac{1}{k_1^2\cos\Theta_2}
p(p-1)|\mu||u_0|^{p-2}\sin(p\arg u_0+\arg\mu-\Theta_2) \\
&=\frac{1}{k_1^2\cos\Theta_2}p(p-1)|\mu||u_0|^{p-2}\sin\Theta_1.
\end{split}
\]
The other values $B(z_0), C(z_0), D(z_0), E(z_0)$ and $F(z_0)$ can be obtained by similar calculation.
\end{proof}

\begin{prop}\label{prop38}
At a singular point $z_0=(u_0,v_0)$ of $P$ with $k_1\ne 0$, 
the determinant of the Hessian $H(z_0)$ is given by
\[
\det H(z_0)=-\frac{1}{k_1^2\cos\Theta_2}4(p-1)(q-1)|\mu|^2 \phi(z_0),
\]
where
\[
\phi(z_0)=(-1)^\kappa(p-1)|v_0|\sin\Theta_3+(q-1)|\mu||u_0|\sin\Theta_1.
\]
\end{prop}

\begin{proof}
By Lemma~\ref{lemma32}, Lemma~\ref{lemma37} and the first equation in Lemma~\ref{lemma31} we have
\[
\begin{split}
(AC-B^2)(z_0)&=-\frac{1}{k_1^2\cos^2\Theta_2}(p-1)^2|\mu|^2, \\
(k_2^2A-2k_2B+C)(z_0)&=\frac{1}{k_1^2}4(p-1)|\mu|^3|u_0|\sin\Theta_1\cos\Theta_2, \\
(DF-E^2)(z_0)&=-\frac{1}{\cos^2\Theta_2}(q-1)^2, \\
(k_4^2D-2k_3k_4E+k_3^2F)(z_0)&=(-1)^\kappa 4(q-1)|v_0|\sin\Theta_3\cos\Theta_4.
\end{split}
\]
Substituting these results to the equation in Lemma~\ref{lemma35}, we have the assertion. 
Note that the first equation in Lemma~\ref{lemma33} implies that $\cos\Theta_2=(-1)^\kappa\cos\Theta_4$
and this equality is used in the the calculation.
\end{proof}

\begin{prop}\label{prop39}
If a singular point $z_0=(u_0,v_0)$ of $P$ satisfies
$k_1\ne 0$ and $\phi(z_0)\ne 0$ then it is a fold.
\end{prop}

\begin{proof}
By Lemma~\ref{lemma31}, $u_0, v_0\in\C\setminus\{0\}$.
Since $k_1=2|\mu|\cos\Theta_1\cos\Theta_2\ne 0$ and $\phi(z_0)\ne 0$ by assumption,
we have $\det H(z_0)\ne 0$ by Proposition~\ref{prop38}, i.e., $\rk H(z_0)=3$.
Thus $z_0$ is a fold by Lemma~\ref{lemma28}.
\end{proof}

\subsection{Cusp singularities in case $k_1\ne 0$ and $\hat R_{\theta_1'\theta_1'}(z_0)\ne 0$}

We assume that $k_1\ne 0$ as in the previous subsection.
Further we assume that $\hat R_{\theta_1'\theta_1'}(z_0)\ne 0$.
The case $\hat R_{\theta_1'\theta_1'}(z_0)=0$ will be studied in the next subsection.

Let $(r_1'', \theta_1'', r_2'', \theta_2'')$ be new coordinates
in a small neighborhood of $z_0$ defined by
\[
(r_1'',\theta_1'', r_2'', \theta_2'')
=(r_1', \theta_1'+\ell_1r_2'+\ell_2\theta_2', r_2', \theta_2'),
\]
where
\[
\ell_1=\frac{\hat R_{\theta_1'r_2'}(z_0)}{\hat R_{\theta_1'\theta_1'}(z_0)}
\quad\text{and}\quad
\ell_2=\frac{\hat R_{\theta_1'\theta_2'}(z_0)}{\hat R_{\theta_1'\theta_1'}(z_0)}.
\]

\begin{lemma}\label{lemma310}
Let $z_0$ be a singular point of $P$ satisfying
$k_1\ne 0$, $\phi(z_0)=0$ and $\hat R_{\theta_1'\theta_1'}(z_0)\ne 0$.
Then the Hessian $H''(z_0)$ of $\hat R$ with variables 
$(\theta_1'', r_2'', \theta_2'')$ at $z_0$ is
\[
H''(z_0)=
\begin{pmatrix}
\hat R_{\theta_1'\theta_1'}(z_0) & 0 & 0 \\
0 & -4(p-1)^2|\mu|^2/(k_1^2\hat R_{\theta_1'\theta_1'}(z_0)) & 0 \\
0 & 0 & 0
\end{pmatrix},
\]
where $\hat R_{\theta_1'\theta_1'}(z_0)=\frac{1}{k_1^2}4(p-1)|\mu|^3|u_0|\sin\Theta_1\cos\Theta_2$.
\end{lemma}

\begin{proof}
Using $\theta_1'=\theta_1''-\ell_1 r_2''-\ell_2\theta_2''$,
we can easily check that 
$\hat R_{\theta_1''\theta_1''}=\hat R_{\theta_1'\theta_1'}$ and 
$\hat R_{\theta_1''r_2''}=\hat R_{\theta_1''\theta_2''}=0$ at $z_0$.
By Lemma~\ref{lemma35}, $\hat R_{\theta_1'\theta_1'}=k_2^2A-2k_2B+C$,
and its value at $z_0$ had been calculated in the proof of Proposition~\ref{prop38}.

By the change of coordinates, 
$\hat R_{r_2''r_2''}=\ell_1^2\hat R_{\theta_1'\theta_1'}-2\ell_1\hat R_{\theta_1'r_2'}+\hat R_{r_2'r_2'}$.
Then the value at $z_0$ is calculated, by using Lemma~\ref{lemma35}, as
\[
\begin{split}
\hat R_{r_2''r_2''}(z_0)
&=\frac{-\hat R_{\theta_1'r_2'}(z_0)^2+\hat R_{r_2'r_2'}(z_0)\hat R_{\theta_1'\theta_1'}(z_0)}{\hat R_{\theta_1'\theta_1'}(z_0)}\\
&=\frac{k_3^2(A(z_0)C(z_0)-B(z_0)^2)+D(z_0)(k_2^2A(z_0)-2k_2B(z_0)+C(z_0))}{\hat R_{\theta_1'\theta_1'}(z_0)}.
\end{split}
\]
By substituting the values in Lemma~\ref{lemma32} and Lemma~\ref{lemma37} and
using the first equation in Lemma~\ref{lemma31} and $\phi(z_0)=0$, we can verify that the numerator 
is $-\frac{4}{k_1^2}(p-1)^2|\mu|^2$.

By the change of coordinates, $\hat R_{r_2''\theta_2''}=\ell_1\ell_2\hat R_{\theta_1'\theta_1'}
-\ell_1\hat R_{\theta_1'\theta_2'}
-\ell_2\hat R_{\theta_1'r_2'}+\hat R_{r_2'\theta_2'}$.
Then the value at $z_0$ is calculated, by using Lemma~\ref{lemma35}, as
\[
\begin{split}
\hat R_{r_2''\theta_2''}(z_0)
=&\frac{1}{\hat R_{\theta_1'\theta_1'}(z_0)}
\left(-\hat R_{\theta_1'\theta_2'}(z_0)\hat R_{\theta_1'r_2'}(z_0)
+\hat R_{\theta_1'\theta_1'}(z_0)\hat R_{r_2'\theta_2'}\right)(z_0)\\
=&\frac{1}{\hat R_{\theta_1'\theta_1'}(z_0)}\{-k_3k_4(k_2A(z_0)-B(z_0))^2\\
&+(k_3k_4A(z_0)+E(z_0))(k_2^2A(z_0)-2k_2B(z_0)+C(z_0))\}.
\end{split}
\]
We can verify that the numerator vanishes by applying lemmas and $\phi(z_0)=0$ as in the previous case.
Thus $\hat R_{r_2''\theta_2''}(z_0)=0$.
Finally we have $\hat R_{\theta_2''\theta_2''}(z_0)=0$ 
since $\hat R_{r_2''r_2''}\ne 0$ and $\det H''(z_0)=0$.
\end{proof}

Thus the new coordinates $(r_1'', \theta_1'', r_2'', \theta_2'')$ satisfies the setting of 
Lemma~\ref{lemma28}. Recall that in these coordinates
\[
\xi_4= -\frac{\partial Q}{\partial \theta_2^{''}}\frac{\partial }{\partial r_1^{''}}+
\frac{\partial Q}{\partial r_1^{''}}\frac{\partial }{\partial \theta_2^{''}},
\]
where the functions and derivatives are taken on the subset $S_1(P)\cap U$ of $S_1(P)$.

\begin{prop}\label{prop311}
Let $z_0$ be a singular point of $P$ satisfying
$k_1\ne 0$, $\phi(z_0)=0$ and $\hat R_{\theta_1'\theta_1'}(z_0)\ne 0$.
Then 
\[
\frac{\partial}{\partial \theta_2^{''}}(\xi_4(\xi_4(\hat R)))(z_0)
= \frac{1}{\sin^2\Theta_1\cos^2\Theta_1\cos^2\Theta_2}\psi(z_0),
\]
where $\psi(z_0)$ is a polynomial with variables $\sin\Theta_i$ and $\cos\Theta_i$, $i=1,2,3$, 
which is not constant zero and does not depend on the value $|\mu|$.
\end{prop}

\begin{proof}
It is easy to verify that the value of $k_2(k_2A-B)-(k_2^2A-2k_2B+C)$ at 
$z_0$ is $0$. Since this appears as a factor of the numerators of
\[
k_2\ell_1-k_3=\frac{k_2\hat R_{\theta_1'r_2'}(z_0)-k_3\hat R_{\theta_1'\theta_1'}(z_0)}{\hat R_{\theta_1'\theta_1'}(z_0)}
\]
and
\[
k_2\ell_2-k_4=\frac{k_2\hat R_{\theta_1'\theta_2'}(z_0)
-k_4\hat R_{\theta_1'\theta_1'}(z_0)}{\hat R_{\theta_1'\theta_1'}(z_0)},
\]
we have $k_2\ell_1-k_3=k_2\ell_2-k_4=0$.
Because of this property, 
the change of coordinates from $(r_1,\theta_1,r_2,\theta_2)$
to $(r_1'',\theta_1'',r_2'',\theta_2'')$ is given as
\[
(r_1,\theta_1,r_2,\theta_2)=\left(\frac{1}{k_1}r_1''-\frac{k_2}{k_1}\theta_1'',\;\;
\theta_1''-\ell_1r_2''-\ell_2\theta_2'',\;\; r_2'',\;\; \theta_2''\right).
\]
We can verify that
\[
\ell_2=\frac{(-1)^\kappa |v_0|\sin\Theta_3}{|\mu||u_0|\sin\Theta_1},
\]
and if furthermore $\phi(z_0)=0$ then $\ell_2= -\frac{q-1}{p-1}$.

By checking the change of coordinates, we see that
\[
\frac{\partial }{\partial \theta_2^{''}}= -\ell_2\frac{\partial }{\partial \theta_1}+\frac{\partial }{\partial \theta_2}
\quad\text{and}\quad
\frac{\partial }{\partial r_1^{''}}= \frac{1}{k_1}\frac{\partial }{\partial r_1}.
\]
By Lemma \ref{lemma31}, $r_1$ is a constant on $S_1(P)\cap U$ and 
\[
   \theta_1= \frac{q-1}{p-1}\theta_2+ \frac{2}{p-1}(\kappa\pi- \arg\mu).
\]
Hence we have
\[
   \frac{\partial }{\partial \theta_2^{''}}=\frac{2(q-1)}{p-1}\frac{\partial }{\partial \theta_1}
   \quad\text{and}\quad
   \frac{\partial }{\partial r_1^{''}}=0.
\]


Now we calculate the equation in the assertion.
We can replace $\hat R$ in the equation by $R$ since
\[
\begin{split}
\xi_4(\hat R)
&=-Q_{\theta_2''}\hat R_{r_1''}+Q_{r_1''}\hat R_{\theta_2''}
=-Q_{\theta_2''}(R_{r_1''}-sQ_{r_1''})+Q_{r_1''}(R_{\theta_2''}-sQ_{\theta_2''}) \\
&=-Q_{\theta_2''}R_{r_1''}+Q_{r_1''}R_{\theta_2''}
= \xi_4(R).
\end{split}
\]
We can easily check that 
\[
   \xi_4=\frac{\partial Q}{\partial r_1^{''}}\frac{\partial }{\partial \theta_2^{''}}
=\frac{4|\mu|(q-1)}{k_1(p-1)}c(\theta_1;\arg\mu)\frac{\partial }{\partial \theta_1},
\]
where
\[
c(\theta_1;\arg\mu)=\cos\left(\frac{p+1}{2}\theta_1\right)\cos\left(\frac{p-1}{2}\theta_1+\arg\mu\right).
\]
Thus we have 
\[
   \frac{\partial}{\partial \theta_2^{''}}(\xi_4(\xi_4(\hat R)))
=\frac{32|\mu|^2}{k_1^2}\left(\frac{q-1}{p-1}\right)^3
   \frac{\partial }{\partial \theta_1}\left(c(\theta_1;\arg\mu)\frac{\partial }{\partial \theta_1}\left(
c(\theta_1;\arg\mu)\frac{\partial R}{\partial \theta_1}\right)\right).
\]
We calculate the right-hand side and then take the value at $z_0=(u_0,v_0)$.
In particular, by this substitution, we have $c(\arg u_0;\arg\mu)=\cos\Theta_1\cos\Theta_2$. 
After these calculation and substitution, we eliminate $|\mu|$ by using $\phi(z_0)=0$. 
In this process, $\sin^2\Theta_1$ appears in the denominator. 
Since $k_1^2$ is also in the denominator, $\cos^2\Theta_1\cos^2\Theta_2$ also appears.
No other $\sin\Theta_i$'s and $\cos\Theta_i$'s appear in the denominator.
Thus we have the equation claimed in the assertion.

By direct calculation, we can see that $\phi(z_0)$ is not zero at the point defined by either
(i) $\sin\Theta_2=\sin\Theta_4=0$ and $\Theta_1=\Theta_3=\frac{\pi}{4}$ or
(ii) $\sin\Theta_2=\sin\Theta_4=0$ and $\Theta_1=-\Theta_3=\frac{\pi}{4}$.
Hence $\phi(z_0)$ is not constant zero.
\end{proof}

\begin{prop}\label{prop312}
For a generic choice of $\mu$, every singular point $z_0$ of $P$ with $k_1\ne 0$
and $\hat R_{\theta_1'\theta_1'}(z_0)\ne 0$ is either a fold or a cusp.
\end{prop}

\begin{proof}
We first prove that $\phi(z_0)=0$ and $\psi(z_0)=0$ have no common solution for a generic choice of $\mu$. 
We will use Chebyshev polynomials to represent these two functions. 
For $n\in \N$, let $T_n(t)$ and $U_n(t)$ be the Chebyshev polynomials defined by
\[
   \cos(n\theta)= T_n(\cos \theta)
   \quad\text{and}\quad 
   \sin((n+1)\theta)= U_n(\cos\theta)\sin\theta.
\]
By the definitions of $\Theta_i$'s and Lemma~\ref{lemma33}, 
\[
\Theta_1=\frac{p+1}{2}\arg u_0
\quad\text{and}\quad
\Theta_3= \frac{q+1}{2}\arg v_0= \frac{q+1}{2}\left( \frac{p-1}{q-1}\arg u_0+ \frac{2(\arg \mu-k\pi)}{q-1}\right).
\]
Since $|u_0|$ and $|v_0|$ are non-zero constant along the singular set $S(P)$,
by setting $\theta=\frac{\arg u_0}{2(q-1)}$ we parametrize $S(P)$ by $\theta$.
Substituting this to the above $\Theta_1$ and $\Theta_3$ we have
\[
\Theta_1=(p+1)(q-1)\theta
\quad\text{and}\quad
\Theta_3=(p-1)(q+1)\theta+\frac{(q+1)(\arg \mu-\kappa\pi)}{q-1}.
\]
We set $n_1,n_3\in\Z$ and $m_\kappa\in\R$ such that $\Theta_1=n_1\theta$ and $\Theta_3=n_3\theta+m_{\kappa}$,
where $m_{\kappa}$ depends on $\arg\mu$.
Set $t=\cos\theta$. 
Then $\phi(z_0)$ is written as
\[
\begin{split}
\phi(z_0)=&(-1)^\kappa(p-1)|v_0|\left(\sin\theta \cos m_\kappa \;U_{n_3-1}(t)+\sin m_\kappa\;T_{n_3}(t) \right) \\
&+(q-1)|u_0||\mu|\sin \theta\;U_{n_1-1}(t).
\end{split}
\]
Similarly, $\psi(z_0)$ is written as
\[
   \psi(z_0)= \sin\theta \;P_1(t;\arg\mu)+ P_2(t;\arg\mu),
\]
where $P_i$ are some polynomials on $t$ whose coefficients depend on $\arg \mu$. 
Remark that $\phi(z_0)$ and $\psi(z_0)$ are not constant zero.
From the equations $\phi(z_0)=0$ of $\theta$, we obtain that
\[
   \Phi(t;\mu)=|\mu|^2\Phi_1(t)+ |\mu|\cos m_\kappa\;\Phi_2(t)+ \cos^2 m_\kappa\;\Phi_3(t)+ \Phi_4(t)=0
\]
and from $\psi(z_0)=0$ we get $\Psi(t;\arg\mu)=0,$
where $\Phi_i$ are non-zero polynomials not depending on $\mu$ and $\Psi$ is a non-zero polynomial on $t$ 
whose coefficients depend on $\arg\mu$. If $\phi(z_0)=0$ and $\psi(z_0)=0$ have a common solution then $\Phi(t;\mu)=0$ 
and $\Psi(t;\arg\mu)=0$ also.

Assume that the four polynomials $\Phi_i(t)$ have common factors
which provide solutions of $\Phi(t;\mu)=0$.
These solutions do not depend $\mu$ since $\Phi_i(t)$ do not depend.
However, since $\sin\Theta_1\ne 0$, any solution of $\phi(z_0)=0$ depends on $|\mu|$.
This means that the above solutions do not satisfy $\phi(z_0)=0$.

Dividing $\Phi(t;\mu)=0$ by these factors,
we may assume that  $\Phi_i(t)$ do not have a common factor. 
Let $S(|\mu|, \arg\mu)$ be the resultant of $\Phi(t;\mu)$ and $\Psi(t;\arg\mu)$ with respect to $t$. 
We will show that $S(|\mu|, \arg\mu)\not\equiv 0$. Assume $S(|\mu|, \arg\mu)\equiv 0$, 
that means the polynomials $\Phi(t;\mu)$ and $\Psi(t;\arg\mu)$ have common factor for any $(|\mu|,\arg\mu)$. 
Choose $\arg\mu$ such that $\cos^2 m_\kappa\;\Phi_3+ \Phi_4$ is not dividable by any non-constant common 
factor of $\Phi_1, \Phi_2$. 
Since $\Phi(t;\mu)$ is dividable by some non-constant factor of $\Psi(t;\arg\mu)$ for all $|\mu|$, 
there are infinitely many $|\mu|$ such that   
$\Phi(t;\mu)$ share same factor $\Delta(t)$. This implies that $\Delta(t)$ is a common factor of $\Phi_1, \Phi_2$ 
and hence of $\cos ^2m_\kappa\; \Phi_3+ \Phi_4$.  
This is a contradiction. Thus $\phi(z_0)=0$ and $\psi(z_0)=0$ have no common solution for generic $\mu$.

Let $z_0$ be a singular point of $P$ with $k_1\ne 0$ and $\hat R_{\theta_1'\theta_1'}(z_0)\ne 0$.
Let $\mu$ be a generic value in $\C$ chosen in the above paragraph.
Furthermore, we may assume that $\sin \Theta_1=0$ and $\sin \Theta_3=0$ 
have no common solution. By direct calculation we have
\[
\hat R_{r_1''\theta_2''}
=(k_2\ell_2-k_4)A-\ell_2B+\frac{1}{k_1}\hat R_{r_1\theta2}.
\]
As we mentioned in the proof of Proposition~\ref{prop311},
$k_2\ell_2-k_4=0$. The term $\hat R_{r_1\theta2}$ also vanishes.
Note that $k_1=2|\mu|\cos\Theta_1\cos\Theta_2\ne 0$ and 
$\hat R_{\theta_1'\theta_1'}\ne 0$ by assumption.
Therefore, $\hat R_{r_1''\theta_2''}=-\ell_2B$.
By substituting the results in Lemma~\ref{lemma32}, Lemma~\ref{lemma35} and Lemma~\ref{lemma37}, we have
\[
\begin{split}
\hat R_{r_1''\theta_2''}(z_0)
&=-\frac{\hat R_{\theta_1'\theta_2'}(z_0)}{\hat R_{\theta_1'\theta_1'}(z_0)}
\frac{(p-1)|\mu|\cos\Theta_1}{k_1\cos\Theta_2} 
=-\frac{k_4(k_2A(z_0)-B(z_0))}{\hat R_{\theta_1'\theta_1'}(z_0)}
\frac{(p-1)|\mu|\cos\Theta_1}{k_1\cos\Theta_2} \\
&=\frac{k_4}{\hat R_{\theta_1'\theta_1'}(z_0)}\frac{2(p-1)^2|\mu|^3\cos\Theta_1}{k_1^3\cos\Theta_2}.
\end{split}
\]
Recall that $k_4=-2|v_0|\sin\Theta_3\cos\Theta_4$.
If $\sin\Theta_3=0$ then $\phi(z_0)=0$ implies $\sin\Theta_1=0$. This cannot happen because we
had chosen $\mu$ such that they have no common solution.
Thus  $\hat R_{r_1''\theta_2''}\ne 0$.
With this non-vanishing and the entries of the matrix $H''(z_0)$
in Lemma~\ref{lemma310}, we can conclude that
the rank of the right-top $3\times 3$ minor of the Hessian of
$\hat R$ with variables $(r_1'',\theta_1'',r_2'',\theta_2'')$ is $3$.
In particular, the rank of the matrix $M$ in Lemma~\ref{lemma28} is $3$.

Now we apply Lemma~\ref{lemma28}.
If $\phi(z_0)\ne 0$ then $z_0$ satisfies $\det H(z_0)\ne 0$ by Proposition~\ref{prop38}. 
Hence it is a fold by Lemma~\ref{lemma28}. 
If $\phi(z_0)=0$ then $\psi(z_0)\ne 0$. This is equivalent to 
$\frac{\partial}{\partial \theta_2^{''}}(\xi_4(\xi_4(\hat R)))(z_0)
\ne 0$ under the assumption $k_1\ne 0$ and $\phi(z_0)=0$.
By Lemma~\ref{lemma310}, $\rk H''(z_0)=2$.
Moreover, as mentioned above, the rank of the matrix $M$ in Lemma~\ref{lemma28} is $3$.
Hence $z_0$ is a cusp by Lemma~\ref{lemma28}.
This completes the proof.
\end{proof}

\subsection{Non-existence of cusps with $\hat R_{\theta_1'\theta_1'}(z_0)=0$ for generic $\mu$}

\begin{lemma}\label{lemma313}
For a generic choice of $\mu$,
any singular point $z_0$ of $P$ with $k_1\ne 0$ and $\phi(z_0)=\hat R_{\theta_1'\theta_1'}(z_0)=0$ is a fold.
\end{lemma}

\begin{proof}
Since $\hat R_{\theta_1'\theta_1'}(z_0)=0$ and $k_1\ne 0$,
we have $\sin\Theta_1=0$. Then $\phi(z_0)=0$ implies $\sin\Theta_3=0$. 
Hence the last equation in Lemma~\ref{lemma33} has no solution if we choose $\arg\mu$ generic,
that is, any singular point $z_0$ satisfies $\phi(z_0)\ne 0$ for generic $\mu$.
By Proposition~\ref{prop38}, $\det H(z_0)\ne 0$. Hence $z_0$ is a fold by Lemma~\ref{lemma28}.
\end{proof}

\subsection{Non-existence of cusps with $k_1=0$ for generic $\mu$}

\begin{lemma}\label{lemma314}
For a generic choice of $\mu$,
any singular point $z_0$ of $P$ with $k_1=0$
is a fold.
\end{lemma}

\begin{proof}
Since $k_1=2|\mu|\cos\Theta_1\cos\Theta_2$, 
either $\cos\Theta_1$ or $\cos\Theta_2$ is zero.

First we assume that $\cos\Theta_2\ne 0$.
In this case, $k_2\ne 0$ because $\cos\Theta_1=0$ and $\sin\Theta_1=\pm 1$.
Let $(r_1',\theta_1', r_2', \theta_2')$ be new coordinates in a small neighborhood of $z_0$ defined by
\[
(r_1',\theta_1', r_2', \theta_2')=
(r_1, k_1r_1+k_2\theta_1+k_3r_2+k_4\theta_2, r_2, \theta_2).
\]
Then we have
$(Q_{r_1'}(z_0), Q_{\theta_1'}(z_0), Q_{r_2'}(z_0), Q_{\theta_2'}(z_0))=(0,1,0,0)$.
Set $\hat R=R-s Q$, where $s=R_{\theta_1'}(z_0)$, so that 
$\hat R_{\theta_1'}=R_{\theta_1'}-s Q_{\theta_1'}$ vanishes at $z_0$.
The constant $s$ is calculated as
\[
\begin{split}
s&=
\frac{1}{k_2}Q_{\theta_1}(z_0)=\frac{1}{k_2}\{p(|\mu||u_0|^p
\cos(p\arg u_0+\arg\mu)-|\mu||u_0|\cos(-\arg u_0+\arg\mu)\} \\
&=\frac{1}{k_2}|u_0|\{\cos(p\arg u_0+\arg\mu)-\cos(-\arg u_0+\arg\mu)\}\\
&=-\frac{2}{k_2}|u_0|\sin\left(\frac{p-1}{2}\arg u_0+\arg\mu\right)\sin\left(\frac{p+1}{2}\arg u_0\right) \\
&=\frac{\sin\Theta_2}{\cos\Theta_2}.
\end{split}
\]
This coincides with the $s$ in the case $k_1\ne 0$ obtained in Lemma~\ref{lemma36}.
Therefore, $\hat R$ above is exactly same as in the case $k_1\ne 0$.

Since $k_1=0$, by straightforward calculation, we have
\[
\hat R_{r_1'r_2'}=-\frac{k_3}{k_2}\hat R_{r_1\theta_1}\quad\text{and}\quad
\hat R_{r_1'\theta_2'}=-\frac{k_4}{k_2}\hat R_{r_1\theta_1}.
\]
We can check that $\hat R_{r_1\theta_1}(z_0)=(p-1)|\mu|\frac{\cos\Theta_1}{\cos\Theta_2}=0$.
Therefore both $\hat R_{r_1'r_2'}$ and $\hat R_{r_1'\theta_2'}$ vanish at $z_0$
and the Hessian of $\hat R$ with variables $(r_1', r_2', \theta_2')$ at $z_0$ is of the form
\[
H(z_0)=
\begin{pmatrix}
\hat R_{r_1'r_1'}(z_0) & 0 & 0 \\
0 & \hat R_{r_2'r_2'}(z_0) & \hat R_{r_2'\theta_2'}(z_0) \\
0 & \hat R_{r_2'\theta_2'}(z_0) & \hat R_{\theta_2'\theta_2'}(z_0) 
\end{pmatrix}.
\]
By straightforward calculation we have 
\[
\hat R_{r_1'r_1'}(z_0)=\hat R_{r_1r_1}(z_0)=p(p-1)|\mu||u_0|^{p-2}\frac{\sin\Theta_1}{\cos\Theta_2}\ne 0.
\]
We also have
\[
\begin{split}
\hat R_{r_2'r_2'}(z_0)&=\hat R_{r_2r_2}+\frac{k_3^2}{k_2^2}\hat R_{\theta_1\theta_1}
=\frac{1}{k_2^2|v_0|}4|\mu||u_0|\sin\Theta_1\cos\Theta_2\phi'(z_0), \\
\hat R_{r_2'\theta_2'}(z_0)
&=\hat R_{r_2\theta_2}+\frac{k_3k_4}{k_2^2}\hat R_{\theta_1\theta_1}
=\frac{(-1)^\kappa}{k_2^2}4|\mu||u_0|\sin\Theta_1\cos\Theta_2\cos\Theta_3\phi(z_0), \\
\hat R_{\theta_2'\theta_2'}(z_0)
&=\hat R_{\theta_2\theta_2}+\frac{k_4^2}{k_2^2}\hat R_{\theta_1\theta_1}
=\frac{(-1)^{\kappa+1}}{k_2^2}4|\mu||u_0||v_0|\sin\Theta_1\sin\Theta_3\cos\Theta_2\phi(z_0), 
\end{split}
\]
where 
\[
\phi'(z_0)=(-1)^\kappa(q-1)|\mu||u_0|\sin\Theta_1\sin\Theta_3-(p-1)|v_0|\cos^2\Theta_3
\]
and $\phi(z_0)$ is the equation in Proposition~\ref{prop38}.
Then the determinant of the Hessian $H$ is calculated as
\[
\det(H(z_0))=16(q-1)|\mu|^3|u_0|^3\sin\Theta_1\cos\Theta_2\phi(z_0)\hat R_{r_1'r_1'}(z_0).
\]
Therefore $\det(H(z_0))=0$ if and only if $\phi(z_0)=0$.
By the last equation in Lemma~\ref{lemma33} 
and the assumption $\cos\Theta_1=0$, we see that if we fix a value of $\arg\mu$ then
there are only finitely many possible values for $\Theta_3$.
The values $|u_0|$ and $|v_0|$ are determined by $p$ and $q$ as seen in Lemma~\ref{lemma31}.
Therefore if we choose $|\mu|$ generic then $\phi(z_0)=0$ has no solution.
Hence any singular point is a fold for generic $\mu$ by Lemma~\ref{lemma28}.

Next we assume that $\cos\Theta_2=0$.
By Lemma~\ref{lemma33}, we also have $\cos\Theta_4=0$.
Therefore the gradient $(k_1,k_2,k_3,k_4)$ of the real part $Q$ at the singular point $z_0$ of $P$ vanishes.
In this case, we consider the gradient of the imaginary part $R$ at $z_0$;
\[
  (\hat k_1,\hat k_2,\hat k_3,\hat k_4)=(R_{r_1}(z_0),R_{\theta_1}(z_0),R_{r_2}(z_0),R_{\theta_2}(z_0)).
\]
Each element is given as
\[
\begin{array}{lll}
\hat k_1=2|\mu|\cos\Theta_1\sin\Theta_2, & & \hat k_2=-2|\mu||u_0|\sin\Theta_1\sin\Theta_2, \\
\hat k_3=2\cos\Theta_3\sin\Theta_4, & & \hat k_4=-2|v_0|\sin\Theta_3\sin\Theta_4.
\end{array}
\]

Since $\cos\Theta_2=0$, we set $\Theta_2=\ell\pi+\frac{\pi}{2}$ with $\ell\in\Z$.
Substituting this to the defining equation of $\Theta_1$,
we obtain
\begin{equation}\label{eq314}
   \Theta_1=\frac{p+1}{p-1}\left(\ell\pi+\frac{\pi}{2}-\arg\mu\right).
\end{equation}
Therefore, by choosing $\arg\mu$ generic, we may assume that $\cos\Theta_1\ne 0$.

The strategy of the calculation below is exactly same as what we did in Section 3.1.
First we apply the change of coordinates
\[
(r_1',\theta_1', r_2', \theta_2')=
(\hat k_1r_1+\hat k_2\theta_1+\hat k_3r_2+\hat k_4\theta_2, \theta_1, r_2, \theta_2)
\]
so that $(R_{r'_1}(z_0),R_{\theta'_1}(z_0),R_{r'_2}(z_0),R_{\theta'_2}(z_0))=(1,0,0,0)$.
Then we have $Q_{r_1'}(z_0)=\frac{1}{\hat k_1}Q_{r_1}(z_0)=\frac{k_1}{\hat k_1}=0$.
This means that we do not need to apply the change of coordinates of the target space $\R^2$
which we did by setting $\hat R=R-sQ$ between Lemma~\ref{lemma34} and Lemma~\ref{lemma35}.
Then the Hessian $\hat H$ of $Q$ with variables $(\theta_1', r_2', \theta_2')$ is
\[
\hat H=
\begin{pmatrix}
\hat k_2^2\hat A-2\hat k_2\hat B+\hat C & \hat k_3(\hat k_2\hat A-\hat B) & \hat k_4(\hat k_2\hat A-\hat B) \\
\hat k_3(\hat k_2\hat A-\hat B) & \hat k_3^2\hat A+\hat D & \hat k_3\hat k_4\hat A+\hat E \\
\hat k_4(\hat k_2\hat A-\hat B) & \hat k_3\hat k_4\hat A+\hat E & \hat k_4^2\hat A+\hat F
\end{pmatrix},
\]
where
\[
\begin{split}
\hat A&=\frac{1}{\hat k_1^2}Q_{r_1r_1},\quad 
\hat B=\frac{1}{\hat k_1}Q_{r_1\theta_1},\quad 
\hat C=Q_{\theta_1\theta_1} \\
\hat D&=Q_{r_2r_2},\quad 
\hat E=Q_{r_2\theta_2}, \quad 
\hat F=Q_{\theta_2\theta_2},
\end{split}
\]
and its determinant is 
\[
\det H=(\hat k_4^2\hat D-2\hat k_3\hat k_4\hat E+\hat k_3^2\hat F)(\hat A\hat C-\hat B^2)
+(\hat k_2^2\hat A-2\hat k_2\hat B+\hat C)(\hat D\hat F-\hat E^2).
\]
By direct calculation, we obtain
\[
\det \hat H(z_0)=\frac{1}{\hat k_1^2}4(p-1)(q-1)|\mu|^2\sin\Theta_2 \phi(z_0),
\]
where $\phi(z_0)$ is the equation in Proposition~\ref{prop38}.
Since $\cos\Theta_2=0$, it is enough to prove that $\phi(z_0)=0$ has 
no solution for generic $\mu$.
Note that $|u_0|$ and $|v_0|$ are fixed by Lemma~\ref{lemma31}.
Set $\Theta_2=\ell\pi+\frac{\pi}{2}$ with $\ell\in\Z$ as before.
By the first equation in Lemma~\ref{lemma33}, $\Theta_4=(\ell-\kappa)\pi+\frac{\pi}{2}$.
Substituting them to the defining equations of $\Theta_i$'s,
we obtain equation~\eqref{eq314} and
\[
   \Theta_3=\frac{q+1}{q-1}\left((\ell-\kappa)\pi+\frac{\pi}{2}\right).
\]
The number of the possible values of the first term of $\phi(z_0)$ is finite,
while the second term can vary according to $\mu$. Therefore $\phi(z_0)=0$ has no solution for generic $\arg\mu$.
Hence any singular point is a fold for generic $\mu$ by Lemma~\ref{lemma28}.
\end{proof}

\begin{rem}\label{generic_ab}
In summary, for $P(u,v;\mu)$ to be an excellent map, $\mu$ is chosen in the above proofs such that
\begin{itemize}
\item $S(|\mu|,\arg\mu)\ne 0$\;\,(required in the proof of Proposition~\ref{prop312}),
\item $\arg \mu$ is not in a finite set in $S^1$\;\,(required in the case $k_1\ne 0$ and $\hat R_{\theta_1'\theta_1'}(z_0)=0$), and
\item $|\mu|$ is not in a finite set in $\R$\;\,(required in the case $k_1=0$).
\end{itemize}
These conditions can be written by using polar coordinates of $a$ and $b$, according to the change of variables 
explained in the beginning of Section~3.
\end{rem}

\section{Proofs of Theorem~\ref{thm2} and the second assertion in Theorem~\ref{thm1}}

\subsection{Case $ab\ne 0$}

First we study the case $ab\ne 0$. As in the previous section,
we set $P(u,v;\mu)=\mu(u^p+\bar u)+v^q+\bar v$ with $p,q\geq 2$ and $\mu\in\C\setminus\{0\}$.
Then, by Lemma~\ref{lemma31}, 
the set of singular points of $f$ consists of $r$ parallel curves $C_k$, $k=0,\cdots, r-1$,
on the torus $\{(u,v)\in\C^2\mid |u|=A, |v|=B\}$,
each of which is parametrized, with parameter $e^{i\theta}\in S^1$, as
\begin{equation}\label{eq41}
(u,v)=\left(Ae^{\left(\frac{q-1}{r}\theta+c_k\right)i}, Be^{\frac{p-1}{r}\theta i}\right),
\end{equation}
where $r=\gcd(p-1,q-1)$, $A=1/p^{1/(p-1)}$, $B=1/q^{1/(q-1)}$ and $c_k=\frac{1}{p-1}(-2\arg\mu+2\pi k)$.

Assume that $P(u,v;\mu)$ is an excellent map and set the map $P_k:C_k\to\C$ as
\[
\begin{split}
P_k(\theta)&=P\left(Ae^{\left(\frac{q-1}{r}\theta+c_k\right)i}, Be^{\frac{p-1}{r}\theta i}\right)\\
&=\mu\left(A^pe^{\left(\frac{p(q-1)}{r}\theta+pc_k\right)i}+Ae^{-\left(\frac{q-1}{r}\theta+c_k\right)i}\right)
+B^qe^{\frac{q(p-1)}{r}\theta i}+Be^{-\frac{p-1}{r}\theta i}.
\end{split}
\]
Since $P$ is an excellent map, the set of cusps of $P$ on $C_k$ corresponds
to the zero points of $dP_k/d\theta=0$. The left hand side is calculated as
\[
\frac{dP_k}{d\theta}=-2e^{\frac{(p-1)(q-1)}{2r}\theta i}\Phi(\theta)
\]
with
\[
\Phi(\theta)=(-1)^k|\mu|\frac{q-1}{r}A\sin\left(\frac{(p+1)(q-1)}{2r}\theta+\frac{p+1}{2}c_k\right)
+\frac{p-1}{r}B\sin\left(\frac{(p-1)(q+1)}{2r}\theta\right),
\]
where we used the formula
\begin{equation}\label{formula41}
e^{ix}-e^{iy}=2ie^{\frac{x+y}{2}i}\sin\frac{x-y}{2}
\end{equation}
and relations $pA^p=qB^q=1$.

Now we set 
\[
T(\theta;|\mu|)=(-1)^k|\mu|\sin(n\theta+c_k')+C\sin(m\theta),
\]
where $m=\frac{(p-1)(q+1)}{2r}$, $n=\frac{(p+1)(q-1)}{2r}$, $c_k'=\frac{p+1}{2}c_k$ and
$C=\frac{(p-1)B}{r}\frac{r}{(q-1)A}>0$,
and observe how the number of roots of $T(\theta;|\mu|)=0$ changes according to the change of $|\mu|$.
Note that $m>n$ if and only if $p>q$.
Since $T(\theta;|\mu|)$ is a smooth function for variables $(\theta,|\mu|)$,
the number of roots changes only at the points satisfying
\[
\begin{split}
T&=(-1)^k|\mu|\sin(n\theta+c_k')+C\sin(m\theta)=0 \quad\text{and}\\
\frac{dT}{d\theta}&=(-1)^k|\mu|n\cos(n\theta+c_k')+Cm\cos(m\theta)=0.
\end{split}
\]

\begin{lemma}\label{lemma41}
If $p>q$ then
the number of solutions of $T(\theta;|\mu|)=0$ is monotone decreasing according to the parameter $|\mu|\in (0,\infty)$.
\end{lemma}

\begin{proof}
Suppose that $\theta_0$ and $|\mu_0|$ satisfy $T(\theta_0;|\mu_0|)=\frac{dT}{d\theta}(\theta_0;|\mu_0|)=0$.
Set $U(\theta,|\mu|,t)=T(\theta;|\mu|)-t$ and consider the hypersurface given by $U=0$.
Note that $(\theta_0,|\mu_0|,0)$ is a point on this surface.
Then the gradient of $U$ at $(\theta_0,|\mu_0|,0)$ is
\[
\text{grad} U(\theta_0,|\mu_0|,0)=(0,(-1)^k\sin(n\theta_0+c_k'),-1).
\]

First we study the case $(-1)^k\sin(n\theta_0+c_k')\ne 0$.
Assume that $(-1)^k\sin(n\theta_0+c_k')>0$.
Since $(-1)^k|\mu|\sin(n\theta_0+c_k')+C\sin(m\theta_0)=0$, we have
\[
\begin{split}
\frac{d^2 T}{d\theta^2}(\theta_0;|\mu_0|)
&=-(-1)^k|\mu|n^2\sin(n\theta_0+c_k')-Cm^2\sin(m\theta_0) \\
&=(-1)^k|\mu|(m^2-n^2)\sin(n\theta_0+c_k')>0.
\end{split}
\]
This means that the graph $\{U=0\}$ on the $(\theta,t)$-plane $\{|\mu|=0\}$ is downward convex.
Observing the position of the hypersurface $U=0$ near $(\theta_0,|\mu_0|,0)$ 
using the gradient, 
we can conclude that the number of the solutions decreases according to the parameter $|\mu|$.
The assertion in the case $(-1)^k\sin(n\theta_0+c_k')<0$ follows by the same argument.

Next we study the case $\sin(n\theta_0+c_k')=0$.
Since $\frac{dT}{d\theta}(\theta_0;|\mu_0|)=0$, we have
\[
\begin{split}
\frac{d^3 T}{d\theta^3}(\theta_0;|\mu_0|)
&=-(-1)^k|\mu|n^3\cos(n\theta_0+c_k')-Cm^3\cos(m\theta_0) \\
&=(-1)^k|\mu|(m^2-n^2)n\cos(n\theta_0+c_k')\ne 0.
\end{split}
\]
This means that $T(\theta;|\mu_0|)=0$ has a root at $\theta=\theta_0$ with multiplicity $3$.
Set $U'(\theta,|\mu|,t)=\frac{\partial T}{\partial \theta}(\theta,|\mu|)-t$ and consider the hypersurface $U'=0$
near the point $(\theta_0,|\mu_0|,0)$.
Then the gradient of $U'$ at $(\theta_0,|\mu_0|,0)$ is
\[
\text{grad} U'(\theta_0,|\mu_0|,0)=(0,(-1)^kn\cos(n\theta_0+c_k'),-1).
\]
If $(-1)^kn\cos(n\theta_0+c_k')>0$ then $\frac{d^3 T}{d\theta^3}(\theta_0;|\mu_0|)>0$.
By the same argument as before,
the number of the solutions decreases near $(\theta_0,|\mu_0|,T(\theta_0;|\mu_0|))$ according to the parameter $|\mu|$.
It also decreases even if $(-1)^kn\cos(n\theta_0+c_k')<0$ by the same argument.
These mean that the number of solutions of $T(\theta;|\mu|)=0$ decreases near $(\theta_0,|\mu_0|,T(\theta_0;|\mu_0|))$ 
according to $|\mu|$.
\end{proof}

\begin{lemma}\label{lemma42}
Suppose $p=q$.
Then the number of the solutions of $T(\theta;|\mu|)=0$ is $\frac{p^2-1}{r}$
except the finite number of points on $|\mu|=1$ satisfying $\sin c_k'=0$.
\end{lemma}

\begin{proof}
The condition $p=q$ implies that $m=n$ and $C=1$. Therefore the number of solutions changes at the point $|\mu|$ 
satisfying the following two equations:
\[
\begin{split}
T(\theta;|\mu|)&=(-1)^k|\mu|\sin(n\theta+c_k')+\sin(n\theta)=0 \\
\frac{dT}{d\theta}(\theta;|\mu|)&=(-1)^k|\mu|n\cos(n\theta+c_k')+n\cos(n\theta)=0.
\end{split}
\]
From these equations, we have
\[
\begin{split}
0&=(-1)^k|\mu|n\sin(n\theta+c_k')\cos(n\theta)-(-1)^k|\mu|n\cos(n\theta+c_k')\sin(n\theta)\\
&=(-1)^k|\mu|n\sin c_k'.
\end{split}
\]
By substituting this to $T(\theta;|\mu|)=0$ and $\frac{dT}{d\theta}(\theta;|\mu|)=0$,
we obtain $|\mu|=1$. This means that the number of solutions does not change except at the points with 
$|\mu|=1$ and $\sin c_k'=0$.
Since the complement of these points in the $\mu$-plane is connected,
the number of solutions is constant outside these points.
If $|\mu|$ is sufficiently small then the number of solutions is equal to the solution of $\sin(n\theta)=0$.
Since $n=\frac{(p+1)(q-1)}{2r}$, the number of solutions is $\frac{p^2-1}{r}$.
\end{proof}

\begin{proof}[Proof of Theorem~\ref{thm2} in case $ab\ne 0$]
Suppose that $ab\ne 0$. We use the notation in Section~3.

We first proof the assertion in the case $p>q$.
If $|\mu|$ is sufficiently small then the number of cusps of the excellent map $P$ on $C_k$ is equal to 
the solution of $\sin(m\theta)=0$, which is $\frac{(p-1)(q+1)}{r}$.
If $|\mu|$ is sufficiently large then it is equal to 
the solution of $\sin(n\theta+c_k')=0$, which is $\frac{(p+1)(q-1)}{r}$.
Then, by Lemma~\ref{lemma41}, the number of cusps on $C_k$ is between $\frac{(p-1)(q+1)}{r}$
and $\frac{(p+1)(q-1)}{r}$.
Since the singular set of $P$ has $r$ components, the assertion follows.

Next we consider the case $p=q$. In this case, we have $m=n$ and $C=1$.
Lemma~\ref{lemma42} proves the assertion unless $|\mu|=1$ and $\sin c_k'=0$.
Suppose that $|\mu|=1$ and $\sin c_k'=0$.
By the second condition we set $c_k'=\ell\pi$ with $\ell\in\Z$.
If $k+\ell$ is odd then $T(\theta;1)=0$ for any $e^{i\theta}\in S^1$. Therefore $P$ is not an excellent map.
If $k+\ell$ is even, then $T(\theta;1)=2\sin(n\theta)=2\sin\left(\frac{p^2-1}{r}\theta\right)$.
Hence if $P$ is an excellent map then the number of cusps of $P$ on $C_k$ is $\frac{p^2-1}{r}$ for any $k=0,1,\cdots,r-1$.
Since the set of singular points of $P$ has $r$ components, the total number of cusps of $P$ is $p^2-1$.
This completes the proof.
\end{proof}

\subsection{Case $ab=0$}

\begin{proof}[Proof of Theorem~\ref{thm2} in case $ab=0$]
Suppose that either $a$ or $b$ is zero.
Without loss of generality, we may assume that $a\ne 0$ and $b=0$.
In this case, by Lemma~\ref{lemma31}, the singular set of $f$ is parametrized as $(u,v)=(Ae^{i\theta},0)$.
By changing the coefficients of $f$ as before, we obtain the form $P=\mu(u^p+\bar u)+v^q$.
The real and imaginary part $f$ and $g$ of $P$ is given as
\[
\begin{split}
f&=|\mu||u|^p\cos(p\arg u+\arg\mu)+|\mu||u|\cos(-\arg u+\arg\mu)+\text{Re}(v^q), \\
g&=|\mu||u|^p\sin(p\arg u+\arg\mu)+|\mu||u|\sin(-\arg u+\arg\mu)+\text{Im}(v^q).
\end{split}
\]
Following the calculation in Section~3 in this setting, we obtain
\[
\begin{array}{lll}
k_1=2|\mu|\cos\Theta_1\cos\Theta_2, & & k_2=-2|\mu||u_0|\sin\Theta_1\cos\Theta_2, \\
k_3=q\,\text{Re}(v^{q-1}), & & k_4=q\,\text{Im}(v^{q-1}).
\end{array}
\]
The strategy is same as in Section~3.1, though we use the coordinates $(r_1,\theta_1,x_2,y_2)$ instead of
$(r_1,\theta_1,r_2,\theta_2)$, where $v=x_2+iy_2$.
We then apply the change of coordinates as in Section~3.1 into $(r_1',\theta_1',x_2',y_2')$ and 
calculate the Hessian of $\hat R$ at $z_0$ with variables $(\theta_1',x_2',y_2')$ as in Lemma~\ref{lemma35}.

Suppose that $q\geq 3$. Since $v_0=0$ at the singular point $z_0=(u_0,v_0)$, we have $D(z_0)=E(z_0)=F(z_0)=0$.
Therefore we have $\det H(z_0)=0$.
If $P$ is an excellent map, then the right $4\times 3$ minor of the Hessian of $\hat R$ with variables 
$(r_1',\theta_1',x_2',y_2')$ has rank $3$. Since any singularity $z_0$ of $P$ satisfies $\det H(z_0)=0$,
any singularity is a cusp, which is a contradiction. Hence $P$ is also not an excellent map.

Suppose that $q=2$. Since $\gcd(p-1,q-1)=1$, the singular set of $Q$ is connected and parametrized as
$(Ae^{i\theta},0)$ with $e^{i\theta}\in S^1$. 
Applying formula~\eqref{formula41} and $pA^{p-1}=p|u_0|^{p-1}=1$ in the first equation in Lemma~\ref{lemma31}, 
we have
\[
\begin{split}
\frac{d}{d\theta}P(Ae^{i\theta},0)
&=\mu\frac{d}{d\theta}(A^pe^{p\theta i}+Ae^{-\theta i})
=i \mu(pA^pe^{p\theta i}-Ae^{-\theta i})\\
&=-2\mu A e^{\frac{p-1}{2}\theta i}\sin\frac{p+1}{2}\theta.
\end{split}
\]
Therefore if $P$ is an excellent map then it has $p+1$ cusps. Thus the assertion follows.
\end{proof}

\subsection{Proof of the second assertion in Theorem~\ref{thm1}}

We now prove the second assertion in Theorem~\ref{thm1}, which states that
any fold singularity appearing in the proof of the first assertion is indefinite.
Since any excellent map in Theorem~\ref{thm1} has at least one cusp by Theorem~\ref{thm2},
it is enough to show that the fold singularities in a neighborhood of any cusp are indefinite.
We only need to check the case where $k_1\ne 0$ and $\hat R_{\theta_1'\theta_1'}(z_0)\ne 0$ discussed in Section~3.2
because a cusp appears in the proof only in this case.

For a singular point $z_0$ of $P$, let $\tau(z_0)$ be the absolute index of $z_0$,
which is defined as follows: Fix an orientation along the set of singularities.
Then the Hessian of $P$ at $z_0$ defines a symmetric, bilinear form on $L_{z_0}\otimes L_{z_0}$.
Let $i$ be the index of this form. Then the absolute index $\tau(z_0)$ is defined by the maximal number of $i$ and 
the rank of this Hessian minus $i$, see~\cite[p.273]{levine}.

Let $z_0$ be a cusp of $P$. The Hessian at $z_0$ had been calculated in Lemma~\ref{lemma310}
and we see that the rank is $2$ and the index is $1$.
Therefore the absolute index of $z_0$ is $1$.
Now we apply Lemma (2) in~\cite[p.274]{levine}.
Since the dimension $n$ of the source manifold is $4$, the equality $\tau(z_0)=(n-2)/2$ is satisfied.
Therefore the absolute index of a fold in a small neighborhood of $z_0$ is $2$.
Since the rank of the Hessian at a fold is $3$, we see that the index at the fold is neither $0$ nor $3$,
i.e., the fold is indefinite. This completes the proof.\qed

\section{Proof of Theorem~\ref{cor4}}


\begin{proof}[Proof of Theorem~\ref{cor4}]
As in the previous sections, we set $P(u,v;\mu)=\mu (u^2+\bar{u})+v^2+\bar{v}$.
We follow the observation in the beginning of Section~4.1 and~4.2 with the condition $p=q=2$.
Then we can see that $S(P)$ consists of one circle $C_0$.
This circle is parametrized as in equation~\eqref{eq41} with $p=q=2$, $r=1$ and $k=0$.
Then the map $P_0=P|C_0$ is given as
\[
 \begin{split}
   P_0(\theta)&=\mu
  \left(\dfrac{1}{4}e^{(2\theta-4\arg\mu)i}
  +\dfrac{1}{2}e^{-(\theta-2\arg\mu)i} \right)+
   \dfrac{1}{4}e^{2\theta i}+\dfrac{1}{2}e^{-\theta i} \\
&=\dfrac{|\mu|}{4}e^{(2\theta-3\arg\mu)i}
  +\dfrac{|\mu|}{2}e^{-(\theta-3\arg\mu)i}+
  \dfrac{1}{4}e^{2\theta i}+\dfrac{1}{2}e^{-\theta i} \\
  &=\frac{1}{4}\left(|\mu|e^{-3i\arg\mu}+1\right) 
e^{2i\theta}+\frac{1}{2}\left(|\mu|e^{3i\arg\mu}+1 \right) e^{-i\theta}.
\end{split}
\]
Set $K=r_\mu e^{-3i\arg\mu}+1$, then 
\[
 \begin{split}
    P(\theta)
    &=\frac{1}{4}\left(K e^{2i\theta}+2\bar{K} e^{-i\theta}\right) 
     =\frac{|K|}{4}\left(e^{i(2\theta+\arg K)}+2 e^{-i(\theta+\arg K)}\right)\\
    &=\frac{|K|}{4}e^{-\frac{i}{3}\arg K}h\left(\theta+\frac{2}{3}\arg K\right),
   \end{split}
\]
where $h(\theta)=e^{2i\theta}+2 e^{-i\theta}$ is the function in the assertion.
Obviously, $h:S^1\to\R^2$ is injective. Hence $P_0$ is also.
Solving $\frac{dh}{d\theta}=0$,
we can confirm that the cusps appear at $\theta=0$, $\frac{2}{3}\pi$ and $\frac{4}{3}\pi$.
Hence the number of cusps is $3$.
\end{proof}

\section{Examples}

In this section, we give two examples of linear deformations.

\begin{ex}
We consider the map $P(u,v;\mu)=\mu(u^3+\bar u)+v^2+\bar v$, which is a deformation of $u^3+v^2$.
If $\mu$ is generic then $P(u,v;\mu)$ is an excellent map as shown in Theorem~\ref{thm1}.
Figure ~\ref{fig1} is the set of critical values of $P(u,v;\mu)$ with $\arg\mu=0$.
The left figure is in the case $|\mu|<\frac{3\sqrt 3}{2}$ and the right one is when $|\mu|>\frac{3\sqrt 3}{2}$.
By Theorem~\ref{thm2}, if it is an excellent map then the number of cusps is between $4$ and $6$. 
The singularity appearing at $|\mu|=\frac{3\sqrt 3}{2}$ looks like a ``beak to beak''.
\begin{figure}[htbp]
\begin{center}
  \includegraphics[scale=1.0]{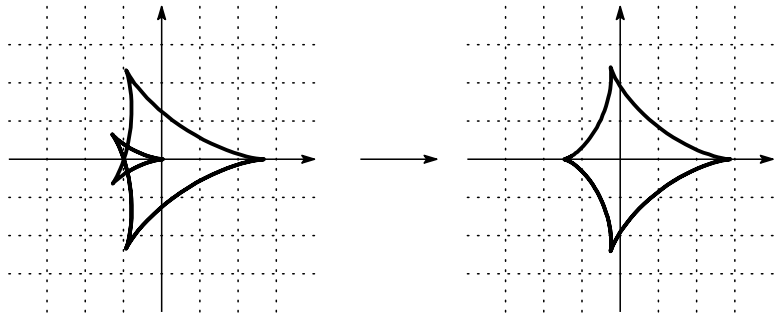}
  \caption{Linear deformations of $u^3+v^2$ into excellent maps.\label{fig1}}
\end{center}
\end{figure}
\end{ex}

\begin{ex}
We consider the map $P(u,v;\mu)=\mu(u^4+\bar u)+v^2+\bar v$, which is a deformation of $u^4+v^2$.
Figure ~\ref{fig2} is the set of critical values of $P(u,v;\mu)$ with $\arg\mu=0$.
The left figure is in the case $|\mu|<c$ and the right one is when $|\mu|>c$, where $c=2.615...$
By Theorem~\ref{thm2}, if it is an excellent map then the number of cusps is between $5$ and $9$. 
The singularities appearing at $|\mu|=c$ look like ``swallowtails''.
\begin{figure}[htbp]
\begin{center}
  \includegraphics[scale=1.0]{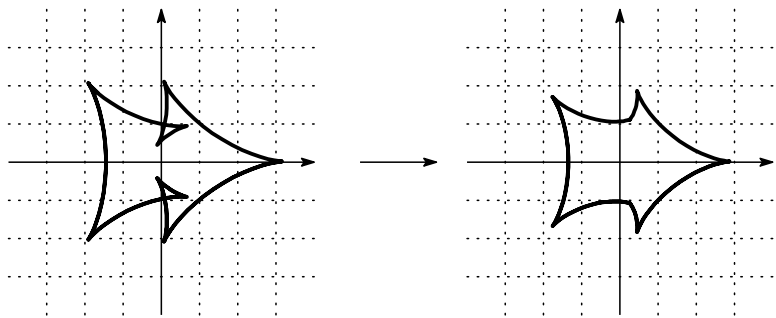}
  \caption{Linear deformation of $u^4+v^2$ into excellent maps.\label{fig2}}
\end{center}
\end{figure}
\end{ex}


%


\end{document}